\documentclass[11pt,reqno]{article}

\usepackage[a4paper,margin=1in]{geometry}
\usepackage{comment}

\usepackage[T1]{fontenc}
\usepackage[utf8]{inputenc} 
\usepackage{lmodern}
\usepackage{microtype}

\usepackage{amsmath,amssymb,amsfonts,amsthm,amscd,mathrsfs}
\usepackage{mathtools}
\numberwithin{equation}{section}

\DeclareMathOperator{\ord}{ord}

\DeclareMathOperator{\spec}{spec}

\usepackage{graphicx}
\usepackage{float}    
\usepackage{enumitem}

\usepackage[ruled,vlined]{algorithm2e} 
\usepackage{circuitikz}

\usepackage{xcolor}
\usepackage[colorlinks=true,
  linkcolor=green!50!black,
  citecolor=green!50!black,
  urlcolor=green!50!black]{hyperref}

\theoremstyle{plain}
\newtheorem{theorem}{Theorem}[section]
\newtheorem{corollary}[theorem]{Corollary}
\newtheorem{lemma}[theorem]{Lemma}
\newtheorem{proposition}[theorem]{Proposition}

\theoremstyle{definition}
\newtheorem{definition}[theorem]{Definition}

\theoremstyle{remark}
\newtheorem{remark}[theorem]{Remark}

\newcommand{\R}{\mathbb{R}}

\newcommand{\Z}{\mathbb{Z}}

\newcommand{\poly}{\mathrm{poly}}
\newcommand{\e}{\mathrm{e}}

\title{Diffusion Computation versus Quantum Computation: A Comparative Model for Order Finding and Factoring}

\author{
Carlos A.\ Cadavid\\
Paulina Hoyos\\
Jay Jorgenson\\
Lejla Smajlovi\'c\\
J.\ D.\ V\'elez\\
}

\date{\today}

\begin{document}
\maketitle


\begin{abstract}\noindent
We study a hybrid computational model for integer factorization in which standard digital arithmetic
is augmented by access to an \emph{iterated diffusion process} on a finite graph.  A \emph{diffusion step}
consists of applying a fixed local-averaging operator (a half-lazy random-walk matrix) to an
$\ell^{1}$-normalized state vector, together with a prescribed readout of a small number of
coordinates.  The goal is not to improve on Shor’s algorithm within the usual circuit models, but to
understand what can be achieved if diffusion is treated as a genuine hardware primitive whose
wall-clock update cost does not scale with the size of the underlying state space.

Let $N\ge 3$ be an odd composite integer with $m\geq 2$ distinct prime factors and let $b\in(\mathbb Z/N\mathbb Z)^{\ast}$ have multiplicative
order $r=\ord_N(b)$.  We attach to the cyclic subgroup $\langle b\rangle$ a weighted Cayley graph with
dyadic generators $b^{\pm 2^t}$ for $0\le t\le M=\lfloor \log_2 N\rfloor+1$, and we analyze the associated
half-lazy walk operator $W$.  Using an explicit character expansion and a doubling lemma for the
lacunary sequence $2^t\bmod r$, we prove that after $
n=O\bigl((\log_2 N)^2\bigr)
$
diffusion steps the single heat-kernel value at the identity satisfies
\[
\left|p_n(e)-\frac{1}{r}\right|\le \frac{1}{4N^2},
\]
so that $r$ is recovered uniquely by rounding $1/p_n(e)$.  Combining this diffusion order-finding
primitive with the classical reduction from factoring to order finding
yields a diffusion-assisted factoring procedure with success probability
$p(m)=1-(m+1)/2^{m}$, depending only on the number $m$ of distinct prime factors of $N$.

In parallel, we develop a complementary relation-finding mechanism: collisions among
dyadic words encountered while exploring the Cayley graph produce loop relations $D=qr$, and repeated gcd updates often stabilize to $r$ (or a small multiple). 
We include numerical examples and implementations, including cycle-based factorizations of the Fermat number
$F_5=2^{32}+1$, of $N=8{,}219{,}999$ (for comparison with recent quantum-annealing experiments),
and of $N=1{,}099{,}551{,}473{,}989$ (for comparison with contemporary quantum computing records).
\end{abstract}

\section{Introduction}
\label{sec:intro}

\subsection{Motivation and scope}

Factoring a composite integer $N$ is a cornerstone problem in computational number theory and cryptography, intimately tied to the structure of $(\Z/N\Z)^\ast$; see, e.g., \cite{HardyWright,CrandallPomerance,MenezesBook,BachShallit}.  Classical methods are highly nontrivial and often subexponential, but no polynomial-time algorithm in $\log_2 N$ is known in the standard digital model \cite{CrandallPomerance}.

Shor's algorithm shows that, in the quantum circuit model, factoring reduces to \emph{order finding} and order finding can be solved in polynomial time via quantum period finding/Fourier sampling \cite{Shor94, Sh97, NC10, EkertJozsa96}. For recent advances and generalizations see \cite{CBK22, Hh25, Ra24, Re25, XQLM23} and references therein.  In broad terms: one encodes modular exponentiation into a coherent linear evolution, extracts periodicity through interference, and completes the computation by classical post-processing (continued fractions and verification) \cite{Shor94,  Sh97, KSV99, NC10}.

This paper investigates a different computational primitive: \emph{diffusion on a finite graph} as an analog mechanism for order finding.  Specifically, we analyze an iterated half-lazy random walk (discrete-time heat flow) on a weighted Cayley graph attached to a cyclic subgroup of $(\Z/N\Z)^\ast$.  Our objective is not to improve on Shor's algorithm within the usual digital/RAM or quantum-circuit models.  Rather, we propose a \emph{hybrid} model in which standard digital arithmetic is augmented by a diffusion process that can be iterated and queried.  In this accounting, order recovery is achieved using $\poly(\log_2 N)$ digital work together with $O((\log_2 N)^2)$ diffusion iterations, under explicit interface assumptions.

\subsection{Quantum layers and diffusion iterations}

For clarity of resource accounting, we compare both non-digital paradigms schematically as alternating
applications of a linear operator and (i) an observation/projection in the quantum case,
and (ii) a classical readout followed by conditioning (post-processing) in the diffusion case:
\[
\begin{aligned}
\text{quantum:}\quad &\cdots \xrightarrow{U}\ \xrightarrow{\text{measure/proj}}\ \xrightarrow{U}\ \xrightarrow{\text{measure/proj}}\ \cdots,\\
\text{diffusion:}\quad &\cdots \xrightarrow{W}\ \xrightarrow{\text{read/cond}}\ \xrightarrow{W}\ \xrightarrow{\text{read/cond}}\ \cdots.
\end{aligned}
\]

This is bookkeeping only: $U$ is unitary (interference-capable), while $W$ is Markovian and dissipative \cite{NC10,LPW09}.
The point is to make explicit where work is assumed to occur in the hybrid model: theorems bound the number of invocations of the diffusion primitive by
$O((\log_2 N)^2)$, while the remaining processing is standard digital arithmetic.

\subsection{Physical diffusion primitive, resource scaling and locality}
Our complexity bounds track two resources separately: digital steps and diffusion steps.

\medskip
\noindent\textbf{Digital steps.}
Deterministic operations on $O(\log_2 N)$-bit integers (modular arithmetic, $\gcd$, verification, bookkeeping), all polynomial in $\log_2 N$ in the standard arithmetic/RAM sense \cite{Knuth2,BachShallit}.

\medskip
\noindent\textbf{Diffusion steps.}
One application of a fixed local-averaging operator, the half-lazy walk operator $W$ on a graph $X$, together with a prescribed readout of a small number of ``temperature'' coordinates of a $\ell^1$-normalized state vector.

The intended interpretation of a diffusion step is hardware-based. Indeed, the speedup occurs only if diffusion is treated as a genuine \emph{hardware primitive}: a single \emph{hardware-level update} that performs local averaging simultaneously at every vertex,
whose wall-clock cost (the actual elapsed real time as measured by a clock) does not scale with the
number of vertices, and whose required readout precision is achievable with $\poly(\log_2 N)$ overhead. In this model, the wall-clock cost of one diffusion step is treated as independent of the
number of vertices represented, because \emph{all vertices are updated in parallel}.
If diffusion is digitally simulated by explicitly updating a length-$r$ state vector (with
$r=\ord_N(b)$ possibly as large as $\varphi(N)$, the Euler totient function), then the advantage disappears and the procedure becomes exponential in $\log_2 N$.

This assumption shifts cost from time to physical resources.  Any literal diffuser must
represent a state over a vertex set of size $\Omega(|X|)$, so device area/energy scales
(at least) linearly with $|X|$.  Our results are therefore conceptual and model-based:
they show that, given such a diffusion primitive with sufficient readout precision,
the order $r$ is determined from a single heat-kernel value after $O((\log_2 N)^2)$
diffusion iterations.

A further operational point is \emph{locality}.  Although the underlying Cayley graph has $r$ vertices, the dynamics are specified by a short generating set and are local; thus an implementation need not pre-construct the full graph.  Instead, the diffusion mechanism can evolve the state by repeated local averaging and (conceptually) ``grow'' only the portion of the graph that is actually reached within the prescribed number of iterations, in the spirit of standard random-walk perspectives on Cayley graphs \cite{LPW09,Chung97}.

\subsection{Main results}

Fix an odd composite integer
\begin{equation}\label{eq:intro:N}
N=\prod_{i=1}^m p_i^{e_i}
\qquad (m\ge 2,\; p_i \text{ distinct odd primes}).
\end{equation}
Let $b\in(\Z/N\Z)^\ast$ have \emph{odd} order $r=\ord_N(b)$.
Without knowing $r$ in advance, we consider a weighted Cayley graph on $\langle b\rangle$ generated by the moves
$b^{\pm 2^t}$ for $0\le t\le M=\lfloor\log_2 N\rfloor+1$.
We analyze the associated half-lazy walk operator $W$ and its discrete-time heat kernel.
We prove that for
\[
n_0 = 4(M+1)\bigl(\log_2 N + 2\bigr)=O((\log_2 N)^2),
\]
the single heat-kernel value at the identity, $p_{n_0}(e)$, determines $r$ uniquely by rounding $1/p_{n_0}(e)$.

There exist deterministic algorithms which ascertain if $N$ is either a prime or the
power of a prime; see, for example, \cite{AKS04}, \cite{Be07}, or \cite{Ra80}.
In these two problems, meaning the determination of $N$ is prime or a prime power, the best known
algorithms have (classical) complexity of order $O((\log N)^{a})$ for some constant $a$.
With this, we do not view the assumption that $N$ is neither a prime nor a prime power as being restrictive, at least from the point of view of theoretical computability.

\subsection{Organization}

The paper is organized as follows.
Section~\ref{sec:prelim} develops the diffusion/heat-kernel framework on weighted graphs and Cayley graphs.
Section~\ref{3} collects the number-theoretic preliminaries used in the reduction from factoring to order finding.
Section~\ref{4} proves the main diffusion order-finding theorem. Section~\ref{5} gives pseudocode for the diffusion-assisted factoring algorithm and for numerical simulation of the diffusion primitive (e.g.\ in Python).
Section~\ref{N5} presents an accelerated, collision-based relation-finding mechanism based on detecting closed cycles and using gcd stabilization of the resulting loop lengths.
Section~\ref{N6} analyzes collision statistics for the walk on $G$, including scale estimates.
Finally, Section~\ref{66} provides illustrative numerical examples.

\section{Diffusion preliminaries}
\label{sec:prelim}

\subsection{Weighted graphs}

Throughout, graphs are finite, undirected, connected, and without self-loops.
A weighted graph $X=(V,E,w)$ consists of a finite vertex set $V$, an edge set $E$ of 2-element subsets of $V$,
and a weight function $w\colon V\times V\to\R$ with the following properties.
\begin{itemize}[leftmargin=2em]
\item \textbf{Symmetry:} $w(x,y)=w(y,x)$ for all $x,y\in V$.
\item \textbf{Non-negativity:} $w(x,y)\ge 0$ for all $x,y\in V$.
\item \textbf{Support on edges:} $w(x,y)>0$ iff $\{x,y\}\in E$.
\end{itemize}
Fix an ordering of $V$ with $|V|=k$. The adjacency operator $A_X$ is the operator on the space of functions
$f$ on $V$ such that $(A_X f)(x)=\sum_{y\in V} w(x,y)\,f(y).$

The degree of a vertex $x \in V$ is
\[
d(x)=\sum_{y\in V} w(x,y).
\]
We call $X$ \emph{regular of degree $d$} if $d(x)=d$ for all $x\in V$.

\subsection{Half-lazy random walks and the discrete-time heat kernel}
\label{subsec:random-walks}

Assume $X$ is $d$-regular. The \emph{half-lazy walk matrix} is
\begin{equation}\label{eq:W}
W = \frac12\left(I+\frac{1}{d}A_X\right).
\end{equation}
Starting from an initial probability distribution $p_0\in\R^k$ ($p_0\ge 0$, $\|p_0\|_1=1$),
define
\[
p_n = W^n p_0.
\]
Write $w_n(x,y)$ for the $(x,y)$-entry of $W^n$. This is the probability that a walk started at $y$ is at $x$ after $n$ steps.
We refer to $p_n$ as the \emph{discrete-time heat kernel}.

Namely, the discrete-time derivative $\partial_n p_n = p_{n+1}-p_n$ and the discrete (probabilistic) Laplacian $\Delta = W-I$ satisfy
\begin{equation}\label{eq:discrete-heat}
\partial_n p_n = \Delta p_n,
\end{equation}
the discrete heat equation on $X$. An explicit expression for the discrete time heat kernel $p_n$ on a regular graph is derived in \cite{CHJSV23}.

\subsection{Spectral expansion and a uniform bound}

Since the graph is $d$-regular, the random-walk matrix is
$P=\frac{1}{d}A_X$.
 Then $P$ is real symmetric and
\[
\spec(P)\subset[-1,1]
\qquad\text{(see e.g.\ \cite[Thm.~7.5]{Ni18}).}
\]
Recall that $ W=\frac12(I+P)$.
If $\theta\in\spec(P)$, then $\frac12(1+\theta)\in\spec(W)$. Hence
\[
\spec(W)=\Bigl\{\tfrac12(1+\theta):\theta\in\spec(P)\Bigr\}\subset[0,1].
\]
In particular, $W$ is symmetric and all its eigenvalues are real and nonnegative.
Let
\begin{equation}
\label{lambdas}
1=\lambda_0>\lambda_1\ge \cdots \ge \lambda_{k-1}\ge 0
\end{equation}
be the eigenvalues of $W$, with an orthonormal eigenbasis
$\{\psi_0,\psi_1,\dots,\psi_{k-1}\}$ satisfying $W\psi_j=\lambda_j\psi_j$.

Moreover, $P\mathbf{1}=\mathbf{1}$ (and hence $W\mathbf{1}=\mathbf{1}$), so we may take $
\psi_0=\frac{1}{\sqrt{k}}\mathbf{1}.
$
Therefore, for any initial distribution $p_0$ we have the spectral expansion
\[
p_n=W^n p_0=\sum_{j=0}^{k-1}\langle \psi_j,p_0\rangle\,\lambda_j^n\,\psi_j,
\]
where $\langle \psi,\phi\rangle$ denotes the usual inner product of vectors $\psi,\phi \in\mathbb{R}^k$.

\begin{proposition}\label{prop:mixing}
Let $\lambda_1$ be the largest eigenvalue of $W$ strictly less than $1$.
If $p_0$ is a probability distribution, then for every vertex $x\in V$ and every $n\ge 0$,
\[
\left|p_n(x)-\frac{1}{k}\right|\le \lambda_1^{\,n}.
\]
\end{proposition}

\begin{proof}
Let $\psi_0=\frac{1}{\sqrt{k}}\mathbf{1}$ and extend this vector to an orthonormal eigenbasis
$\{\psi_0,\psi_1,\dots,\psi_{k-1}\}$ of $W$: $W\psi_j=\lambda_j\psi_j,$ with
$1=\lambda_0>\lambda_1\ge \lambda_2\ge \cdots$.
Since $p_0$ is a probability distribution,
\[
\langle \psi_0,p_0\rangle=\frac{1}{\sqrt{k}}\sum_{x\in V}p_0(x)=\frac{1}{\sqrt{k}}.
\]
Therefore
\[
p_n=W^n p_0=\sum_{j=0}^{k-1}\langle \psi_j,p_0\rangle\,\lambda_j^n\,\psi_j
=\frac{1}{k}\mathbf{1}+\sum_{j=1}^{k-1}\langle \psi_j,p_0\rangle\,\lambda_j^n\,\psi_j.
\]
Set $u_n=\sum_{j=1}^{k-1}\langle \psi_j,p_0\rangle\,\lambda_j^n\,\psi_j$, so $p_n=\frac{1}{k}\mathbf{1}+u_n$
and $u_n$ is orthogonal to $\mathbf{1}$.

Fix $x\in V$ and set $a_j=\langle \psi_j,p_0\rangle$ and $b_j=\lambda_j^{n}\psi_j(x)$ for $1\le j\le k-1$.
Then $u_n(x)=\sum_{j=1}^{k-1} a_j b_j$, and by Cauchy--Schwarz,
\[
|u_n(x)|
\le \Bigl(\sum_{j=1}^{k-1} |a_j|^2\Bigr)^{1/2}\Bigl(\sum_{j=1}^{k-1} |b_j|^2\Bigr)^{1/2}
= \Bigl(\sum_{j=1}^{k-1}|\langle \psi_j,p_0\rangle|^2\Bigr)^{1/2}
\Bigl(\sum_{j=1}^{k-1}\lambda_j^{2n}\psi_j(x)^2\Bigr)^{1/2}.
\]

By Parseval's identity,
\[
\|p_0\|_2^2=\sum_{j=0}^{k-1}|\langle \psi_j,p_0\rangle|^2,
\qquad\text{hence}\qquad
\sum_{j=1}^{k-1}|\langle \psi_j,p_0\rangle|^2\le \|p_0\|_2^2.
\]
Since $\lambda_j\in[0,\lambda_1]$ for $j\ge1$,
\[
\sum_{j=1}^{k-1}\lambda_j^{2n}\psi_j(x)^2
\le \lambda_1^{2n}\sum_{j=1}^{k-1}\psi_j(x)^2
\le \lambda_1^{2n}\sum_{j=0}^{k-1}\psi_j(x)^2.
\]
Finally, using $\langle \psi_j,\delta_x\rangle=\psi_j(x)$ and Parseval's identity for $\delta_x$, we see that
\[
\sum_{j=0}^{k-1}\psi_j(x)^2=\sum_{j=0}^{k-1}|\langle \psi_j,\delta_x\rangle|^2=\|\delta_x\|_2^2=1.
\]
Combining those bounds yields $|u_n(x)|\le \lambda_1^n\|p_0\|_2$.
Since $p_0\ge0$ and $\|p_0\|_1=1$, we have $\|p_0\|_2\le \|p_0\|_1=1$, hence
$|p_n(x)-\tfrac1k|=|u_n(x)|\le \lambda_1^n$.
\end{proof}

\subsection{Weighted Cayley graphs of finite abelian groups}

Let $G$ be a finite abelian group written additively, $S\subseteq G$ a symmetric generating set (so $s\in S\Rightarrow -s\in S$),
and $\alpha\colon S\to\R_{>0}$ with $\alpha(s)=\alpha(-s)$.
The weighted Cayley graph $X=\mathrm{Cay}(G,S,\alpha)$ has vertex set $G$ and edge weights
\[
w(x,y)=\alpha(x-y)\quad\text{if }x-y\in S,\qquad w(x,y)=0\text{ otherwise}.
\]
It is regular of degree
\[
d=\sum_{s\in S}\alpha(s).
\]
Characters diagonalize the adjacency operator $A_X$; see \cite{CR62} for background and \cite[Cor.~3.2]{Ba79} for the Cayley-graph spectral formula.
If $\chi$ is a character of $G$, then it is an eigenvector of $A_X$ with eigenvalue
\[
\eta(\chi)=\sum_{s\in S}\alpha(s)\chi(s),
\]
hence an eigenvector of $W$ with eigenvalue $\lambda(\chi)=\frac12\left(1+\eta(\chi)/d\right)$.

\section{Number-theoretic preliminaries}
\label{3}

This section records the elementary number theory used in the factoring reduction.
All statements below are standard; we include short proofs in the form and level of generality
needed later. General references include \cite[Chs.~1--3]{HardyWright},
\cite[Chs.~2--4]{CrandallPomerance}, and \cite[Chs.~2--4]{BachShallit}; see also
\cite[\S 2--5]{MenezesBook} for standard cryptographic formulations of the order-finding
to factoring reduction.

Let $N$ be a positive odd integer which we write as a product
\begin{equation}\label{eq:nt:N-factor}
N=\prod_{i=1}^{m} p_i^{e_i},
\end{equation}
where $m\ge 2$, the primes $p_1,\dots,p_m$ are distinct and odd, and the exponents satisfy $e_i>0$.
In particular, $N$ is neither prime nor a prime power.

Let $\mathbb{Z}_N=\mathbb{Z}/N\mathbb{Z}$ denote the residue ring modulo $N$, and let
$\mathbb{Z}_N^\ast$ denote its group of units.  There is a natural mapping
\begin{equation}\label{eq:nt:CRT}
\mathbb{Z}_N \longrightarrow
\mathbb{Z}_{p_1^{e_1}}\times \cdots \times \mathbb{Z}_{p_m^{e_m}},
\qquad
a \longmapsto \bigl(a \bmod p_1^{e_1},\dots,a \bmod p_m^{e_m}\bigr).
\end{equation}
By the Chinese Remainder Theorem (CRT), \eqref{eq:nt:CRT} is an isomorphism of rings.  Restricting to units
yields an isomorphism of groups
\begin{equation}\label{eq:nt:gN}
g_N:\mathbb{Z}_N^\ast \xrightarrow{\ \sim\ }
\mathbb{Z}_{p_1^{e_1}}^\ast\times \cdots \times \mathbb{Z}_{p_m^{e_m}}^\ast.
\end{equation}
In a slight abuse of notation, we occasionally use $x$ to denote either an element of $\mathbb{Z}_N^\ast$
or its image $g_N(x)$.

For each $i$, the group $\mathbb{Z}_{p_i^{e_i}}^\ast$ is cyclic under multiplication. Fix a generator $u_i$.
Its order is
\[
\operatorname{ord}_{p_i^{e_i}}(u_i)=\varphi(p_i^{e_i})=p_i^{e_i-1}(p_i-1)=2^{c_i}p_i',
\]
where $c_i>0$ and $p_i'$ is odd.

\medskip
\noindent
\emph{Without loss of generality we assume that the primes are ordered so that}
\begin{equation}
\label{decreciente} c_1\ge c_2\ge \cdots \ge c_m.
\end{equation}

\medskip
The following proposition shows that, assuming a parity condition on the exponents in the CRT decomposition, one can produce a nontrivial square root of $1$ modulo $N.$
\begin{proposition}\label{prop:nt:nontrivial-sqrt}
Let $a\in\mathbb{Z}_N^\ast$.  Let $s\ge 0$ be the \emph{least} integer such that
\[
a^{2^{s}q}\equiv 1 \pmod N
\]
for some odd integer $q$.  Write
\[
g_N(a)=(u_1^{d_1},\dots,u_m^{d_m}).
\]
If there exist indices $i<j$ such that $d_i$ is odd and $d_j$ is even, then $s>0$.
Moreover, setting
\[
x=a^{2^{s-1}q},
\]
one has that
\begin{equation}\label{eq. congr property}
x^2\equiv 1 \pmod N
\qquad\text{and}\qquad
x\not\equiv \pm 1 \pmod N.
\end{equation}
\end{proposition}

\begin{proof}
Assume $d_i$ is odd and $d_j$ is even.  Write $d_j=2^v d_j'$ with $v>0$ and $d_j'$ odd.

Since $a^{2^s q}\equiv 1\pmod N$, it follows that
\[
u_i^{2^s q d_i}\equiv 1 \pmod{p_i^{e_i}}.
\]
Thus the order $\operatorname{ord}_{p_i^{e_i}}(u_i):=2^{c_i}p_i'$ divides $2^s q d_i$.
Because $q$ and $d_i$ are odd, the $2$-part forces $2^{c_i}\mid 2^s$, hence $s\ge c_i\ge 1$.

Set
\[
x=a^{2^{s-1}q}=(u_1^{2^{s-1}q d_1},\dots,u_m^{2^{s-1}q d_m}).
\]
Then
\[
x^2=a^{2^s q}\equiv 1 \pmod N.
\]
By minimality of $s$, we have $x\not\equiv 1\pmod N$.

It remains to show $x\not\equiv -1\pmod N$.
Now
\[
x \equiv u_j^{2^{s-1}q d_j}=u_j^{2^{s-1}q\cdot 2^v d_j'} = u_j^{2^{s+v-1}q d_j'} \pmod{p_j^{e_j}}.
\]
We claim that $x$ is $1$ modulo $p_j^{e_j}$.  Indeed, since $a^{2^s q}\equiv 1\pmod{p_j^{e_j}}$ we have
\[
u_j^{2^s q d_j}=u_j^{2^{s+v}q d_j'}\equiv 1\pmod{p_j^{e_j}}.
\]
Thus the order $2^{c_j}p_j'$ divides $2^{s+v}q d_j'$.  As $q d_j'$ is odd, this implies
\[
s+v\ge c_j
\quad\text{and}\quad
p_j' \mid q d_j'.
\]
Since $s\ge c_i\ge c_j$ (by \eqref{decreciente}) and $v>0$, we have $s+v-1\ge c_j$, hence $2^{c_j}p_j'$ divides
$2^{s+v-1}q d_j'$, which is exactly the condition that
\[
u_j^{2^{s+v-1}q d_j'}\equiv 1\pmod{p_j^{e_j}}.
\]
Therefore $x\equiv 1\pmod{p_j^{e_j}}$.

But this implies that $x\not\equiv -1\pmod N$, for otherwise
 $x\equiv -1 \pmod{p_j^{e_j}}$ also holds. Subtracting the two congruences would give
$2\equiv 0 \pmod{p_j^{e_j}}$, a contradiction because $p_j$ is odd.

\end{proof}

\medskip
\begin{lemma}\label{lem:nt:repetitions}
Let $N$ be as in \eqref{eq:nt:N-factor}, and let $M=\lfloor \log_2 N\rfloor+1$.
For any $a\in\mathbb{Z}_N^\ast$, define the list of elements
\[
S(a)=\{\,a^{2^t}\bmod N:\ t=0,\dots,M\,\}\ \cup\ \{\,a^{-2^t}\bmod N:\ t=0,\dots,M\,\}
\subset \mathbb{Z}_N^\ast.
\]

If there is a repetition in $S(a)$, then (from that repetition) one can determine an odd integer $q$
and the least $s\ge 0$ such that $a^{2^s q}\equiv 1\pmod N$ using at most $O(\log_2 N)$ deterministic steps.

Moreover, if $a$ is chosen uniformly at random from $\mathbb{Z}_N^\ast$, then with probability
\[
p(m)=1-\frac{m+1}{2^m},
\]
the element $x=a^{2^{s-1}q}$ (when $s>0$) satisfies \eqref{eq. congr property}.
\end{lemma}

\begin{proof}
A repetition in $S(a)$ means that for some $0\le t'<t\le M$ and some choice of signs,
\[
a^{2^t}\equiv a^{\pm 2^{t'}}\pmod N,
\]
hence
\[
a^{2^t\pm 2^{t'}}\equiv 1\pmod N.
\]
Factor the exponent as
\[
2^t\pm 2^{t'} = 2^{t'}(2^{t-t'}\pm 1)=2^{t'}q,
\]
where $q=2^{t-t'}\pm 1$ is odd.  From $q$ one can determine the least $s$ such that
$a^{2^s q}\equiv 1\pmod N$ by repeatedly dividing by $2$ when possible, which takes $O(\log_2 N)$ checks.

For the probability bound, write $g_N(a)=(u_1^{d_1},\dots,u_m^{d_m})$.
For each $i$, the exponent $d_i$ is uniformly distributed modulo $2$ (half even, half odd),
and CRT makes the parity vector $(d_1\bmod 2,\dots,d_m\bmod 2)$ uniform on $\{0,1\}^m$.

By Proposition~\ref{prop:nt:nontrivial-sqrt}, success occurs if there exists $i<j$ with $d_i$ odd and $d_j$ even.
Failure means there is \emph{no} such pair, i.e.\ there do not exist indices $i<j$ with
$d_i$ odd and $d_j$ even. Equivalently, the parity vector
\[
(d_1\bmod 2,\dots,d_m\bmod 2)\in\{0,1\}^m
\]
is \emph{nondecreasing} (once a $1$ appears, all later entries are $1$). Hence it must be of the form
\[
(\underbrace{0,\dots,0}_{t\ \text{zeros}},\underbrace{1,\dots,1}_{m-t\ \text{ones}})
\qquad\text{for some }t\in\{0,1,\dots,m\},
\]
giving exactly $m+1$ possibilities. Each possibility occurs with probability $2^{-m}$, so the failure probability is
$(m+1)/2^{m}$ and the success probability is $1-(m+1)/2^m$.
\end{proof}

\begin{remark}\label{rem:nt:amplification}
Since $N$ is neither prime nor a prime power, we have $m\ge 2$.
Hence $p(2)=1-\frac{3}{4}=\frac14,$
and the function $p(m)=1-\frac{m+1}{2^m}$ is strictly increasing for $m\ge 2$ with
$p(m)\to 1$ as $m\to\infty$.
In particular, without knowing $m$ a priori, a single random choice of $a\in\Z_N^\ast$
produces an element $x=a^{2^{s-1}q}$ (when $s>0$) with the desired property \eqref{eq. congr property}
with probability at least $p(2)=\frac14$.
Equivalently, the failure probability satisfies
\[
\Pr[\text{failure in one trial}]=1-p(m)=\frac{m+1}{2^m}\le \frac34.
\]

If we repeat the construction independently $k$ times (fresh random choices of $a$),
then the probability that none of the resulting values $x$ satisfies the two conditions is
\[
(1-p(m))^k \le \left(\frac34\right)^k.
\]
For highly composite $N$ (large $m$), the quantity $1-p(m)=(m+1)2^{-m}$ is much smaller than $3/4$,
so the failure probability decays substantially faster than $(3/4)^k$.
\end{remark}

\medskip
\begin{lemma}\label{lem:nt:cyclic-order}
Let $G$ be a finite cyclic group of even order $n=2^c m$ with $c>0$ and $m$ odd, and let $u$ be a generator.
Then for $1\le d\le n$ one has
\[
\operatorname{ord}_G(u^d)=\frac{n}{\gcd(n,d)}.
\]
\end{lemma}

\begin{proof}
This is standard since $u^d$ generates the subgroup of index $\gcd(n,d)$.
\end{proof}

\medskip
\begin{lemma}\label{lem:nt:oddify}
For any $a\in\mathbb{Z}_N^\ast$, the largest power of $2$ dividing $\operatorname{ord}_N(a)$ is strictly less than $\log_2 N$.
In particular, if $M=\lfloor\log_2 N\rfloor+1$ and $b=a^{2^M}$, then $\operatorname{ord}_N(b)$ is odd.
\end{lemma}

\begin{proof}
Write $g_N(a)=(a_1,\dots,a_m)$ with $a_i\in\mathbb{Z}_{p_i^{e_i}}^\ast$.
Then
\[
\operatorname{ord}_N(a)=\operatorname{lcm}\bigl(\operatorname{ord}_{p_1^{e_1}}(a_1),\dots,\operatorname{ord}_{p_m^{e_m}}(a_m)\bigr).
\]
Each $\operatorname{ord}_{p_i^{e_i}}(a_i)$ divides $\varphi(p_i^{e_i})=2^{c_i}p_i'$, hence the $2$-adic valuation of
$\operatorname{ord}_N(a)$ is at most $\max_i c_i$.

Finally, $2^{c_i}\mid \varphi(p_i^{e_i})<p_i^{e_i}\le N$, so $c_i<\log_2 N$ for all $i$, hence
$\max_i c_i<\log_2 N$.  Taking $M=\lfloor\log_2 N\rfloor+1$ makes $2^M$ divisible by the entire $2$-part of $\operatorname{ord}_N(a)$,
so $b=a^{2^M}$ has odd order.
\end{proof}

\medskip
\begin{proposition}\label{prop:nt:factoring-to-order}
Set $M=\lfloor\log_2 N\rfloor+1$.  For any $a\in\mathbb{Z}_N^\ast$, let $b=a^{2^M}$ and write
\[
r_b=\operatorname{ord}_N(b),
\]
which is odd by Lemma~\ref{lem:nt:oddify}.  If $r_b$ is known, then $\operatorname{ord}_N(a)$ can be computed in at most
$O(\log_2 N)$ deterministic steps.

Moreover, if $a$ is chosen uniformly at random from $\mathbb{Z}_N^\ast$, then with probability at least
$p(m)=1-(m+1)/2^m$ the order $r_a=\operatorname{ord}_N(a)$ is even and the element $x=a^{r_a/2}$ satisfies \eqref{eq. congr property}.
\end{proposition}

\begin{proof}
Since $b=a^{2^M}$ and $r_b$ is odd, the order of $a$ has the form $r_a=2^k r_b$ for some $0\le k\le M$.
One finds the least such $k$ by testing $a^{2^k r_b}\equiv 1\pmod N$ for $k=0,1,\dots,M$, which costs $O(\log_2 N)$
deterministic steps (repeated squaring / modular exponentiation).

For the probability statement, apply Lemma~\ref{lem:nt:repetitions} and Proposition~\ref{prop:nt:nontrivial-sqrt}.
\end{proof}

\section{Diffusion order finding}
\label{4}

This section contains the analytic core: an order--finding primitive based on iterating a
\emph{half--lazy walk operator} on the Cayley graph of the cyclic subgroup generated by $b$.
The method is \emph{local}: one never constructs the full Cayley graph.
A single application of the walk operator at a point $x\in \langle b \rangle$ only requires evaluating the neighbors
\[
  x \longmapsto x\,b^{\pm 2^t}\pmod N \qquad (0\le t\le M),
\]
which can be executed by modular multiplication (and modular inversion).
\smallskip

\noindent\textbf{The underlying Cayley graph and the random walk matrix.}
Let
\[
  M=\lfloor\log_2 N\rfloor+1,
  \qquad
  G=\langle b\rangle\subseteq (\Z/N\Z)^\ast.
\]
Consider the (unweighted) Cayley graph $X=X_{N,b}=\mathrm{Cay}(G,S)$ with generating multiset
\[
  S=\{\,b^{\pm 2^t}: 0\le t\le M\,\}.
\]
Equivalently, $X$ is the weighted graph with vertex set $G$ and weights
\[
  w(x,y)=\#\bigl\{\,t\in\{0,\dots,M\}: y=x\,b^{2^t}\ \text{or}\ y=x\,b^{-2^t}\,\bigr\}.
\]
Then $X$ is $d$--regular of degree $d=2(M+1)$, with adjacency matrix $A_X$ given by
\[
  (A_X p)(x)=\sum_{y\in G} w(x,y)\,p(y).
\]
As in Section~\ref{subsec:random-walks}, the (simple) random-walk matrix on $X$ is
\[
  P=\frac{1}{d}A_X.
\]
In particular, for functions $p\colon G\to\R$ we have the concrete formula
\begin{equation}\label{eq:A-def-mult}
  (Pp)(x)= \frac{1}{2(M+1)}\sum_{t=0}^M
  \Bigl(p(x\,b^{2^t})+p(x\,b^{-2^t})\Bigr),
  \qquad x\in G.
\end{equation}

Fix an ordering $G=\{x_1,\dots,x_{|G|}\}$ and identify $p:G\to\R$ with the column vector
$$(p(x_1),\dots,p(x_{|G|}))^{\mathsf T}\in\R^{|G|}.$$
As before, define the associated half-lazy walk operator
\[
  W=W_{N,b}=\frac12(I+P).
\]

\subsection{The diffusion theorem}

\begin{theorem}[Diffusion order finding]\label{thm:main}
Let $N>1$ be an integer.
Let $b\in(\Z/N\Z)^\ast$ have multiplicative order $r=\ord_N(b).$
Set
\[
  M=\lfloor\log_2 N\rfloor+1,
  \qquad
  G=\langle b\rangle\subseteq (\Z/N\Z)^\ast,
  \qquad |G|=r.
\]
 Let $\{p_n\}_{n\geq 0}$ be a half-lazy walk on $G$ with  transition matrix $W$ and the initial state being the delta function at the identity,
 meaning
\[
  p_0=\delta_e,
  \qquad
  p_n= W^n p_0, \qquad (n\ge 0).
\]
\smallskip
\noindent Then, for every
\begin{equation}\label{eq:diffusion_bound}
  n \;\ge\; 4(M+1)\bigl(\log_2 N + 2\bigr),
\end{equation}
the single readout value $p_n(e)$ determines $r$ uniquely as the unique integer in $(0,N]$
whose reciprocal lies within $1/(4N^2)$ of $p_n(e)$.
\end{theorem}

Theorem \ref{thm:main} will be proved in subsection \ref{subsec. proof of main thm}.

\begin{corollary}\label{cor:rounding-rule}
Under the hypotheses of Theorem~\ref{thm:main}, if $|p_n(e)-1/r|\le 1/(4N^2)$ then
\[
\operatorname{round}\!\bigl(1/p_n(e)\bigr)=r.
\]
\end{corollary}

\begin{proof}
Since $r\le N$, we have
\[
p_n(e)\ge \frac{1}{r}-\frac{1}{4N^2}\ge \frac{1}{N}-\frac{1}{4N^2}=\frac{4N-1}{4N^2}>\frac{3}{4N}.
\]
Therefore
\[
\left|\frac{1}{p_n(e)}-r\right|
=\left|\frac{r(1/r-p_n(e))}{p_n(e)}\right|
\le \frac{N\cdot (1/(4N^2))}{3/(4N)}=\frac13<\frac12,
\]
which forces rounding to equal $r$.
\end{proof}
\begin{remark}\label{rem:scope-diffusion-thm}
Theorem~\ref{thm:main} is stated in full generality. It only concerns diffusion on the cyclic subgroup
$G=\langle b\rangle$ of order $r=\ord_N(b)$, and its proof does not use any special arithmetic
hypotheses on $N$ beyond the explicit inequalities stated.
In the remainder of this paper, however, we will invoke Theorem~\ref{thm:main} only in the
factoring setting, where $N$ is assumed to be an odd integer that is neither prime nor a prime
power, and where we choose $b$ so that $r=\ord_N(b)$ is odd.
\end{remark}
\medskip\noindent
Next, we will state and prove a \emph{doubling lemma} we will need in the proof of
Theorem~\ref{thm:main} in the next section. Conceptually, this is a lacunary-series phenomenon: the doubling sequence $2^t$
spreads residues quickly enough that one must encounter an interval where the phase has nonpositive cosine value, reminiscent of the mechanisms
exploited in Korobov-type bounds for lacunary exponential sums, see e.g. \cite{KM12} or \cite{Va19}.

\begin{lemma}[Doubling Lemma]\label{lem:middle-half}
Let $r$ be an integer and let $M\ge 0$ be such that $2^M>r$.
For every integer $k$ with $1\le k\le r-1$, there exists $t\in\{0,1,\dots,M\}$ such that
\[
  k2^t \bmod r \in \Bigl[\frac r4,\frac{3r}{4}\Bigr].
\]
\end{lemma}

\begin{proof}
For each $t\ge 0$, let $a_t\in\{0,1,\dots,r-1\}$ be the least residue of $k2^t$ modulo~$r$.
Assume for contradiction that for all $t\in\{0,\dots,M\}$,
\[
  a_t\in \Bigl[0,\frac r4\Bigr)\ \cup\ \Bigl(\frac{3r}{4},r\Bigr).
\]
Set
\[
  b_t:=\min(a_t,\ r-a_t).
\]
Then $0<b_t<r/4$ for all $t\in\{0,\dots,M\}$.

We claim that $b_{t+1}=2b_t$ for $t=0,\dots,M-1$.
If $a_t<r/4$, then $2a_t<r/2<r$, so $a_{t+1}=2a_t$ and hence $b_{t+1}=2b_t$.
If $a_t>3r/4$, then $2a_t\in(3r/2,2r)$, so $a_{t+1}=2a_t-r\in(r/2,r)$ and therefore
\[
  b_{t+1}=r-a_{t+1}=r-(2a_t-r)=2(r-a_t)=2b_t.
\]
Thus $b_t=2^t b_0$ for all $t\le M$. Since $b_0=\min(k,r-k)\ge 1$, we have $b_M\ge 2^M$.

By hypothesis $2^M>r$, hence $b_M\ge 2^M>r/4$, contradicting $b_M<r/4$.
This contradiction proves the lemma.
\end{proof}

\medskip\noindent
Recall that for $k\in\{0,1,\dots,r-1\},$ if
\[
  \chi_k(j)=\exp\!\left(\frac{2\pi i k j}{r}\right),
  \qquad j\in\mathbb{Z}/r\mathbb{Z},
\]
then
\[
  \frac{1}{r}\sum_{k=0}^{r-1}\chi_k(j)=
  \begin{cases}
    1, & j\equiv 0\pmod r,\\
    0, & j\not\equiv 0\pmod r.
  \end{cases}
\]

\subsection{Proof of Theorem~\ref{thm:main}}\label{subsec. proof of main thm}

\begin{proof}
We change to additive notation. The subgroup $G=\langle b\rangle$ is cyclic of order $r$. Hence,
\[
  \phi:\mathbb{Z}/r\mathbb{Z}\longrightarrow G,
  \qquad
  \phi(j)=b^j
\]
is an isomorphism. Using $\phi$ we transport functions $p:G\to\R$ to functions
$\widetilde p:\Z/r\Z\to\R$ by $\widetilde p(j)=p(b^j)$. In particular $\widetilde p(0)=p(e)$.
Under this identification, right multiplication by $b^{\pm 2^t}$ becomes translation by $\pm 2^t$ modulo $r$.
Hence the random walk operator $P$ becomes the translation-invariant walk
\[
  (P\widetilde p)(j)
  =\frac{1}{2(M+1)}\sum_{t=0}^M\Bigl(\widetilde p(j+2^t)+\widetilde p(j-2^t)\Bigr),
  \qquad j\in\mathbb{Z}/r\mathbb{Z}.
\]

Now the characters are eigenfunctions of $P$. Since $\chi_k(j\pm a)=\chi_k(j)\chi_k(\pm a)$ one has
\[
  (P\chi_k)(j)
  =\chi_k(j)\cdot \frac{1}{2(M+1)}\sum_{t=0}^M\left(e^{2\pi i k2^t/r}+e^{-2\pi i k2^t/r}\right)
  =\chi_k(j)\cdot \frac{1}{M+1}\sum_{t=0}^M \cos\!\left(\frac{2\pi k2^t}{r}\right).
\]
Thus $P\chi_k=\mu_k\chi_k$ with
\[
  \mu_k=\frac{1}{M+1}\sum_{t=0}^M \cos\!\left(\frac{2\pi k2^t}{r}\right),
\]
and therefore $W\chi_k=\lambda_k\chi_k$ where $\lambda_k=(1+\mu_k)/2$.
As noted before (\ref{lambdas}), $\lambda_0=1$ and $0\le \lambda_k\le 1$ for all $k$.

Let $\delta_0$ be the delta function at $0\in\Z/r\Z$. By orthogonality,
\[
  \delta_0(j)=\frac{1}{r}\sum_{k=0}^{r-1}\chi_k(j).
\]
Applying $W^n$ and evaluating at $0$ gives
\[
  (W^n\delta_0)(0)=\frac{1}{r}\sum_{k=0}^{r-1}\lambda_k^n
  =\frac{1}{r}+\frac{1}{r}\sum_{k=1}^{r-1}\lambda_k^n.
\]
Transporting back to $G$, the left-hand side is exactly $p_n(e)$, hence
\begin{equation}\label{eq:pn-error}
  \left|p_n(e)-\frac{1}{r}\right|
  \le \max_{1\le k\le r-1}\lambda_k^n.
\end{equation}

Fix $k\in\{1,\dots,r-1\}$. By Lemma~\ref{lem:middle-half} there exists $t\in\{0,\dots,M\}$ such that
\[
  k2^t \bmod r \in \Bigl[\frac r4,\frac{3r}{4}\Bigr],
  \qquad\text{hence}\qquad
  \cos\!\left(\frac{2\pi k2^t}{r}\right)\le 0.
\]
All other cosine terms are $\le 1$, so
\[
  \mu_k\le \frac{M}{M+1}=1-\frac{1}{M+1},
  \qquad\text{and therefore}\qquad
  \lambda_k=\frac{1+\mu_k}{2}\le 1-\frac{1}{2(M+1)}.
\]
Combining with \eqref{eq:pn-error} and using $1-x\le \e^{-x}$,
\[
  \left|p_n(e)-\frac{1}{r}\right|
  \le \left(1-\frac{1}{2(M+1)}\right)^n
  \le \exp\!\left(-\frac{n}{2(M+1)}\right).
\]

If $n\ge 4(M+1)(\log_2 N+2)$ then
\[
\exp\!\left(-\frac{n}{2(M+1)}\right)\le \exp\!\bigl(-2(\log_2 N+2)\bigr)
= \e^{-4}\,\e^{-2\log_2 N}.
\]
Since $\log_2 N=\ln N/\ln 2\ge \ln N$, we have $\e^{-2\log_2 N}\le \e^{-2\ln N}=N^{-2}$,
so the right-hand side is $\le \e^{-4}N^{-2}<\frac{1}{4N^2}$. Hence
\[
  \left|p_n(e)-\frac{1}{r}\right|\le \frac{1}{4N^2}.
\]

\medskip\noindent\textbf{Uniqueness of $r$.}  Since $r=\ord_N(b)$ divides $|\Z_N^\ast|=\varphi(N)$ and $N>1$, we have
$\varphi(N)<N,$ hence $r<N$. (In particular $2^M>N>r,$ for $M=\lfloor\log_2 N\rfloor+1.$)
For distinct integers $a\neq b$ in $\{1,\dots,N\}$,
\[
  \left|\frac{1}{a}-\frac{1}{b}\right|
  =\frac{|a-b|}{ab}\ge \frac{1}{N^2}.
\]
Therefore there is at most one integer $r'\in(0,N]$ such that
$\left|p_n(e)-\frac{1}{r'}\right|\le \frac{1}{4N^2}$.
The preceding estimate shows that $r$ satisfies this inequality, hence $r$ is uniquely determined by $p_n(e)$.
\end{proof}

\section{A diffusion-assisted factoring algorithm}
\label{5}

In this section, we describe diffusion assisted factorization algorithm. It receives as an input a positive integer $N$ with $m\geq 2$ odd prime factors. The algorithm returns a divisor $d$ of $N$, with probability at least $p(m)=1-2^{-m}(m+1)$ for each choice of a random integer $a\in\{1,2,\ldots, N-1\}$.

Before we proceed with the pseudocode for the algorithm, let us state the following two remarks.

\begin{remark}\label{rem:algo-comments}
Algorithm~\ref{alg:diffusion-factoring} is probabilistic in the same sense as Shor's factoring algorithm:
a single \emph{trial} (i.e.\ one choice of $a$ and the subsequent steps) may end in a restart
(\texttt{continue} in Algorithm~\ref{alg:diffusion-factoring}; equivalently, \textsf{FAIL} for that particular $a$),
for instance because the extracted square root satisfies $x\equiv \pm 1\pmod N$ and hence yields only a trivial gcd.

However, for composite $N$ with at least two distinct prime factors, each independent trial has a nonzero success
probability depending only on the CRT parity pattern (cf.\ Proposition~\ref{prop:nt:factoring-to-order}
and Lemma~\ref{lem:nt:repetitions}). Consequently, repeating the procedure independently drives the overall
failure probability very close to $0$.

We will leave for elsewhere the problem of optimizing
the probability $p(m)$ of success. In that regard, the methodology of \cite{Za13} seems applicable.
\end{remark}

\medskip

\begin{remark}
\label{rem:primality-primepower}
In practice one can efficiently exclude the cases ``$N$ is prime'' and ``$N$ is a prime power''
before invoking Algorithm~\ref{alg:diffusion-factoring}. For primality testing there are fast randomized tests
(e.g.\ Miller--Rabin), deterministic polynomial-time algorithms (AKS), and practical certificate-based methods
(e.g.\ ECPP). To test whether $N$ is a prime power, one may first perform a \emph{perfect-power test}
(decide whether $N=u^k$ with $k\ge 2$) and, if so, apply a primality test to the base~$u$.
See, for example, \cite{CrandallPomerance}.
\end{remark}

\begin{algorithm}[H]
\caption{\textsc{Diffusion-Assisted Factoring}$(N)$}
\label{alg:diffusion-factoring}
\DontPrintSemicolon
\KwIn{Odd composite $N\ge 3$ that is neither prime nor a prime power.}
\KwOut{A nontrivial factor $d$ of $N$.}

\While{true}{
  \BlankLine
  \textbf{(1) Random choice and gcd test.}\;
  Choose $a\gets$ uniform in $\{1,2,\dots,N-1\}$.\;
  $d\gets \gcd(a,N)$.\;
  \If{$1<d<N$}{\Return $d$}

  \BlankLine
  \textbf{(2) Compute the doubling multiset.}\;
  $M\gets \lfloor \log_2 N\rfloor+1$.\;
  Compute
  \[
    S(a)=\{a^{\pm 2^t}\bmod N:\ t=0,1,\dots,M\}
  \]
  by repeated squaring (and inversion).\;

  \BlankLine
  \textbf{(3) Early collision branch.}\;
  \If{$S(a)$ contains a repetition}{
    From the repetition, deterministically extract an odd $q$ and the least $s\ge 0$
    such that $a^{2^s q}\equiv 1\pmod N$ (Lemma~\ref{lem:nt:repetitions}).\;
    \If{$s>0$}{
      $x\gets a^{2^{s-1}q}\bmod N$.\;
      $d\gets \gcd(x-1,N)$.\;
      \If{$1<d<N$}{\Return $d$}
    }
  }

  \BlankLine
  \textbf{(4) Oddify the order.}\;
  $b\gets a^{2^M}\bmod N$.\; \tcp*[r]{then $\ord_N(b)$ is odd (Lemma~\ref{lem:nt:oddify})}

  \BlankLine
  \textbf{(5) Diffusion order finding.}\;
  Use the diffusion primitive (Theorem~\ref{thm:main}) to recover
  $r_b\gets \ord_N(b)$.\;

  \BlankLine
  \textbf{(6) Lift order and extract a factor.}\;
  Find the least $k\in\{0,1,\dots,M\}$ such that $a^{2^k r_b}\equiv 1\pmod N$.\;
  $r_a\gets 2^k r_b$.\;
  \If{$r_a$ is even}{
    $x\gets a^{r_a/2}\bmod N$.\;
    $d\gets \gcd(x-1,N)$.\;
    \If{$1<d<N$}{\Return $d$}
  }
  \textbf{continue} \tcp*[r]{\textsf{FAIL} for this $a$; restart with new $a$}
}
\end{algorithm}

\section{Relation finding via collisions and gcd stabilization}
\label{N5}

This section discusses a complementary (faster) mechanism for recovering the order $r=\ord_N(b)$
from explicit relations produced by collisions while exploring the Cayley graph of
\[
G=\langle b\rangle \subseteq (\Z/N\Z)^\ast.
\]
We view the Cayley graph as being explored from the identity by following short words in the dyadic
generating multiset
\[
S=\{\,b^{\pm 2^t}:\ 0\le t\le M\,\},
\qquad M=\lfloor \log_2 N\rfloor+1.
\]
Whenever two distinct sampled words land at the same vertex, we obtain an arithmetic
relation in the group.   Repeated collisions yield many such relations;
taking gcds of the corresponding exponents often stabilizes to $r$ (or to a small multiple of $r$).

We note that two such words combine to form a closed geodesic on the graph,
so the problem in hand is closely related to measuring the girth of the graph and
detecting its shortest closed non-trivial path.  As a result, we are showing another manner
in which the geometry of the graphs we consider are related to the arithmetic associated to $\ord_N(b)$.


\subsection{Exponent words and loop relations}
\label{subsec:loops}

As in the proof of Theorem~\ref{thm:main} (and throughout the paper), we work in exponent
coordinates. Assume $r$ is odd and identify
\[
G=\langle b\rangle \simeq \Z/r\Z,\qquad j\longmapsto b^j .
\]
A (non-lazy) word of length $L$ is a sequence
\[
w=\bigl[(\varepsilon_1,t_1),\dots,(\varepsilon_L,t_L)\bigr],
\qquad
\varepsilon_i\in\{\pm1\},\ \ t_i\in\{0,\dots,M\}.
\]
Its integer exponent and endpoint are
\[
E(w)=\sum_{i=1}^L \varepsilon_i 2^{t_i}\in\Z,
\qquad
x(w)=b^{E(w)}\in G .
\]
Clearly,
\begin{equation}\label{eq:E-bound-rel}
|E(w)|\le L\,2^M.
\end{equation}

\begin{definition}[Word collision and loop relation]\label{def:word-collision-loop}
Let $w,w'$ be words of length at most $L$.
We say that $(w,w')$ is a \emph{word collision} if
\[
x(w)=x(w') \quad \text{in } G.
\]
The associated \emph{loop difference} (or \emph{loop exponent}) is
\[
D(w,w')=E(w)-E(w').
\]
A word collision is called \emph{nontrivial} if $E(w)\neq E(w')$, equivalently if $D(w,w')\neq 0$.
\end{definition}

\begin{proposition}\label{prop:loop-multiple-clean}
Let $w,w'$ be words of length at most $L$ forming a \emph{nontrivial} word collision
(Definition~\ref{def:word-collision-loop}). Then
\[
D(w,w') \neq 0,\qquad r\mid D(w,w'),
\qquad \text{and}\qquad 0<|D(w,w')|\le 2L\,2^M.
\]
Equivalently, there exists $q\in\Z\setminus\{0\}$ such that
\[
D(w,w')=q\,r
\qquad\text{and}\qquad
|q\,r|\le 2L\,2^M.
\]
\end{proposition}

\begin{proof}
The collision $x(w)=x(w')$ means
\[
b^{E(w)}\equiv b^{E(w')}\pmod N,
\]
hence
\[
b^{E(w)-E(w')}\equiv 1\pmod N.
\]
By definition of $r=\ord_N(b)$, this implies $r\mid D(w,w')$.

If $E(w)=E(w')$, then $D(w,w')=0$ and the collision yields no relation; by assumption we are in
the nontrivial case, so $D(w,w')\neq 0$.

Finally, by \eqref{eq:E-bound-rel},
\[
|D(w,w')|\le |E(w)|+|E(w')|\le 2L\,2^M.\qedhere
\]
\end{proof}

\begin{remark}[Running gcd is always a multiple of $r$]\label{rem:gcd-multiple}
If $D_1,\dots,D_s$ are nonzero loop differences coming from nontrivial word collisions, then each
$D_i$ is divisible by $r$, hence
\[
r \mid g_s=\gcd(|D_1|,\dots,|D_s|).
\]
Thus the running gcd can only decrease as more relations are collected, and it always remains a
multiple of the true order.
\end{remark}


\subsection{Gcd stabilization and the zeta function}
\label{subsec:gcd-zeta-clean}

Each collision yields, by Proposition~\ref{prop:loop-multiple-clean}, a nonzero relation $D_i = q_i r \neq 0,$ $i=1,\ldots, s$ and we form the gcd $g_s=\gcd(|D_1|,\dots,|D_s|).$
Clearly,
\[
g_s=r\cdot \gcd(|q_1|,\dots,|q_s|).
\]
In particular, $g_s=r$ if and only if $\gcd(q_1,\dots,q_s)=1$.

Thus, order recovery reduces to understanding how quickly $\gcd(q_1,\dots,q_s)$ drops to $1$.

\begin{remark}\label{heur:indep-clean}
In favorable long walks, distinct collisions arise from essentially unrelated word pairs,
and the multipliers $q_i=D_i/r$ behave like approximately independent nonzero integers in a
comparable magnitude range.
\end{remark}
The following Theorem is well known in the literature (see \cite{Nym72} and \cite{Leh00}).

\begin{theorem}[Riemann's zeta law for gcds]\label{thm:zeta-gcd-clean}
Fix $s\ge 2$.  Let $U_1,\dots,U_s$ be independent random integers, each uniform on $\{1,2,\dots,Q\}$.
Then as $Q\to\infty$,
\[
\Pr(\gcd(U_1,\dots,U_s)=1)= \frac{1}{\zeta(s)}+O\left(\frac{\delta_{s=2} \log Q+1 }{Q}\right),
\]
where $\delta_{s=2}=1$ if $s=2$, and it equals zero otherwise.
\end{theorem}

\begin{remark}\label{rem:zeta-numerics-clean}
\[
\frac{1}{\zeta(2)}=\frac{6}{\pi^2}\approx 0.608,\qquad
\frac{1}{\zeta(3)}\approx 0.832,\qquad
\frac{1}{\zeta(4)}=\frac{90}{\pi^4}\approx 0.924.
\]
This explains why only \emph{ a few independent
relations} often suffice for the running gcd to drop to the true order.
\end{remark}

\section{Collision bounds for the walk on $G$}
\label{N6}

We specialize the collision discussion to our (half--lazy) random walk on the cyclic group
\[
G=\langle b\rangle,\qquad |G|=r,
\]
with transition operator $W$ as in the previous sections. For $x,y\in G$ and $t\ge 0$, define the
$t$--step transition probabilities by
\[
p(y,x;t)=(W^t\delta_y)(x).
\]
Thus $p(y,x;t)$ is the probability that a walk started at $y$ is at $x$ after $t$ steps.  In
particular, the endpoint distribution at time $t$ for a walk started at the identity $e$ is
\[
p_t(x)=p(e,x;t)=(W^t\delta_e)(x),\qquad x\in G.
\]


To avoid dependence between samples taken along a single trajectory, we generate endpoints by
\emph{independent walks}: we run the walk for $t$ steps starting from $e$, record the endpoint,
restart at $e$, and repeat this procedure $T$ times.

\begin{definition}\label{def:omega-coordinates}
Fix $t\ge 0$ and an integer $T\ge 1$. Set $\Omega=G^T$.
For $\mathbf{x}=(x_1,\dots,x_T)\in \Omega$, define the coordinate maps
\[
X_i:\Omega\to G,\qquad X_i(\mathbf{x})=x_i\qquad (i=1,\dots,T).
\]
\end{definition}

\begin{definition}\label{def:iid-walk}
Assume Definition~\ref{def:omega-coordinates}. We say that $X_1,\dots,X_T$ are sampled by
\emph{independent restarts at time $t$} if for every $\mathbf{x}=(x_1,\dots,x_T)\in\Omega$,
\[
\Pr\bigl((X_1,\dots,X_T)=(x_1,\dots,x_T)\bigr)=\prod_{i=1}^T p(e,x_i;t).
\]
\end{definition}


\begin{definition}
  Assume Definition \ref{def:iid-walk}. We say that a \emph{collision} occurs in the sample  $X_1,\dots,X_T$ if there exist $1\le i<j\le T$ and $x\in G$, so that $ X_i (x)=X_j(x)$.
\end{definition}

\begin{definition}\label{def:s2}
Define the \emph{collision parameter} at time $t$ by
\[
s_2(t)=\sum_{x\in G}p(e,x;t)^2=\sum_{x\in G}p_t(x)^2.
\]
\end{definition}


\begin{theorem}\label{thm:birthday-scale}
Let $X_1,\dots,X_T$ be sampled by independent restarts at time $t$
(Definition~\ref{def:iid-walk}). Let $Z$ denote the number of colliding pairs among the $T$ samples, i.e.
\[
Z=\#\{(i,j):1\le i<j\le T,\ X_i(x)=X_j(x)\text{  for some  } x\in G\}.
\]
The expected value of the first collision among $T$ independent walks
at time $t$ is
\[
\mathrm E[Z]=\binom{T}{2}\,s_2(t).
\]
In particular, as $t\to\infty$, so that $x\mapsto p(e,x;t)$ is close to the uniform distribution on $G$, we
have that
\begin{equation}\label{eq. expentation collisions}
\mathrm E[Z] \sim \frac{T(T-1)}{2r}.
\end{equation}
\end{theorem}
\begin{proof}
The random variables $X_1,\dots,X_T$ are independent and
\[
\Pr(X_i=x)=p_t(x)=p(e,x;t), \qquad (x\in G,\,\, i=1,\ldots,T).
\]
Hence, the expected value $\mathrm E[Z]$ is
\[
\mathrm E[Z]=\sum_{1\le i<j\le T}\Pr( X_i(x)=X_j(x)\text{  for some  } x\in G).
\]
Fix $i<j$. Using independence and the common law $p_t$, we have
\[
\Pr(X_i=X_j)
=\sum_{x\in G}\Pr(X_i=x,\ X_j=x)
=\sum_{x\in G}\Pr(X_i=x)\Pr(X_j=x)
=\sum_{x\in G}p_t(x)^2
=s_2(t).
\]
There are $\binom{T}{2}$ choices of $(i,j)$, hence \eqref{eq. expentation collisions} holds true.

Finally, if $t$ is large enough that $p_t$ is close to uniform on $G$ (so $p_t(x)\approx 1/r$ for
all $x\in G$), then
\[
s_2(t)=\sum_{x\in G}p_t(x)^2 \approx \sum_{x\in G}\frac{1}{r^2}=\frac{1}{r},
\]
which proves the second statement.
\end{proof}

\begin{remark}
Heuristically speaking, Theorem \ref{thm:birthday-scale} tells us that the first collision among $T$ independent walks at time $t$ becomes plausible (meaning that the expectation of the collision is close to $1$) once
\[
\binom{T}{2}\,s_2(t)\approx 1,
\qquad\text{equivalently}\qquad
T\approx \frac{1}{\sqrt{s_2(t)}}.
\]
In particular, if $t$ is large enough that $x\mapsto p(e,x;t)$ is close to uniform on $G$, then
\[
s_2(t)\approx \frac{1}{r},
\qquad\text{and hence}\qquad
T\approx \sqrt{r}.
\]

\end{remark}

\begin{remark}
This heuristic explained above immediately separates two regimes for a \emph{purely digital} collision search
(i.e.\ when the walk is simulated and endpoints are generated by ordinary arithmetic, rather than
by a diffusion primitive).  In the generic situation one expects
\[
r=\ord_N(b)\ \text{to be large, typically comparable to }\varphi(N)\ (\text{and often of order }N).
\]
Then the collision scale becomes
\[
T\approx \sqrt{r}\ \approx\ \sqrt{\varphi(N)},
\]
so the number of restarts required to see a first collision is itself on the order of a square
root of the ambient group size.  This is precisely what one expects from a classical digital
algorithm whose only mechanism for producing relations is collision detection: it is governed by
statistics and does not yield a polynomial-time order-finding procedure in $\log N$.

On the other hand, if the chosen base happens to have \emph{small order}
\[
r=\ord_N(b)\ll \varphi(N),
\]
then the same estimate predicts a substantial digital speedup:
\[
T\approx \sqrt{r}\ \ll\ \sqrt{\varphi(N)}.
\]
In this situation the method can be quite practical on an ordinary digital computer, because the
sampling burden is reduced from a square root of $\varphi(N)$ down to a square root of $r$, and the
subsequent gcd-stabilization step typically needs only a small number of essentially independent
relations.  This behavior is visible in numerical experiments; see Example~\ref{E3}, where $r$ is
small enough that collisions appear after a manageable number of independent restarts and the
order is recovered quickly by standard digital processing.

It is worth separating this \emph{digital} speedup from what changes (and what does not) in the
diffusion-based model.  When the diffusion primitive of Section~\ref{4} is available, the order is not
recovered by waiting for random-walk endpoints to collide.  Instead, one reads a single heat-kernel
value after a controlled number of diffusion iterations: the analysis shows that after
$t=O((\log_2 N)^2)$ updates the distribution has flattened enough that $p_t(e)\approx 1/r$, and this
single scalar already determines $r$ by rounding.  Thus the scale $T\approx \sqrt{r}$ is not
the quantity governing performance in the diffusion order-finding procedure.

By contrast, if one tries to realize the \emph{collision method itself} in a diffusion device, then the
digital estimate $T\approx \sqrt{r}$ still describes how many essentially independent samples must be
generated to see collisions. Replacing $\sqrt{\varphi(N)}$ by $\sqrt{r}$ can be a major improvement in software, but it may have
limited impact on a physical diffuser, where the dominant costs are often dictated by the size,
energy, and precision required to represent and evolve a state whose underlying vertex set has
on the order of $r$ elements.

Finally, one can envision hybrid variants that use diffusion hardware more directly for collision-style
relation finding.  Rather than restarting only at the identity $e$, one could initiate diffusion from
many starting points simultaneously, allowing the Cayley graph to be explored in parallel as the mass
spreads outward from multiple sources.  In that picture, ``collisions'' would manifest as overlaps of
expanding profiles (heat fronts) rather than as literal coincidences of two sampled endpoints.  We do
not pursue such multi-source growth strategies here, but they suggest a natural direction in which
diffusion parallelism might be leveraged beyond the single-source setup analyzed above.
\end{remark}

\section{Examples}
\label{66}
This section illustrates the diffusion--based order recovery on different composite integers
as well as examples employing the digital collision strategy.

\subsection{Example 1: $N=299$ }
\label{subsec:example-299}

Take $N=299=13\cdot 23$ and $b=3\in(\Z/N\Z)^\ast.$
One checks that $\gcd(b,N)=1$. The element $b$ has \emph{odd} multiplicative order modulo $N$:
\[
  r=\ord_{299}(3)=33.
\]
As in Section 4, we form the Cayley graph on the cyclic subgroup
$\langle b\rangle\subset(\Z/N\Z)^\ast$ using generators $b^{\pm 2^t}$ for $0\le t\le M$ with $M=\lfloor \log_2 N\rfloor+1$.
See Figure \ref{fig:cayley-299}. Then we  run the half--lazy random walk started at the identity element.
\begin{figure}[H]
  \centering
  \includegraphics[width=0.5\textwidth]{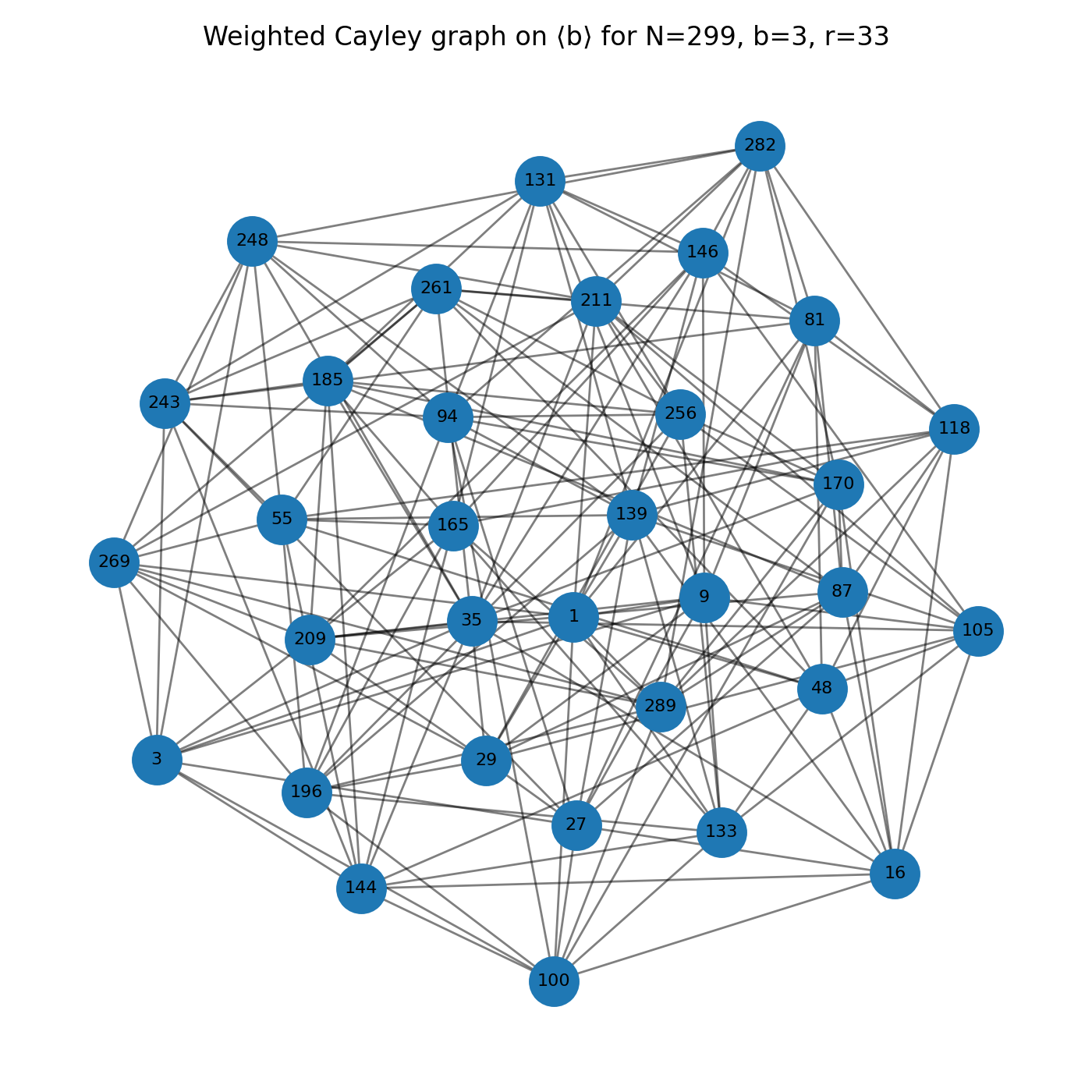}
  \caption{Cayley graph on $\langle 3\rangle\subset(\Z/299\Z)^\ast$ with generators $b^{\pm 2^t}$.
  Here $N=299$, $b=3$, and $\ord_{299}(3)=33$.}
  \label{fig:cayley-299}
\end{figure}

\begin{figure}[H]
  \centering
  \includegraphics[width=0.5\textwidth]{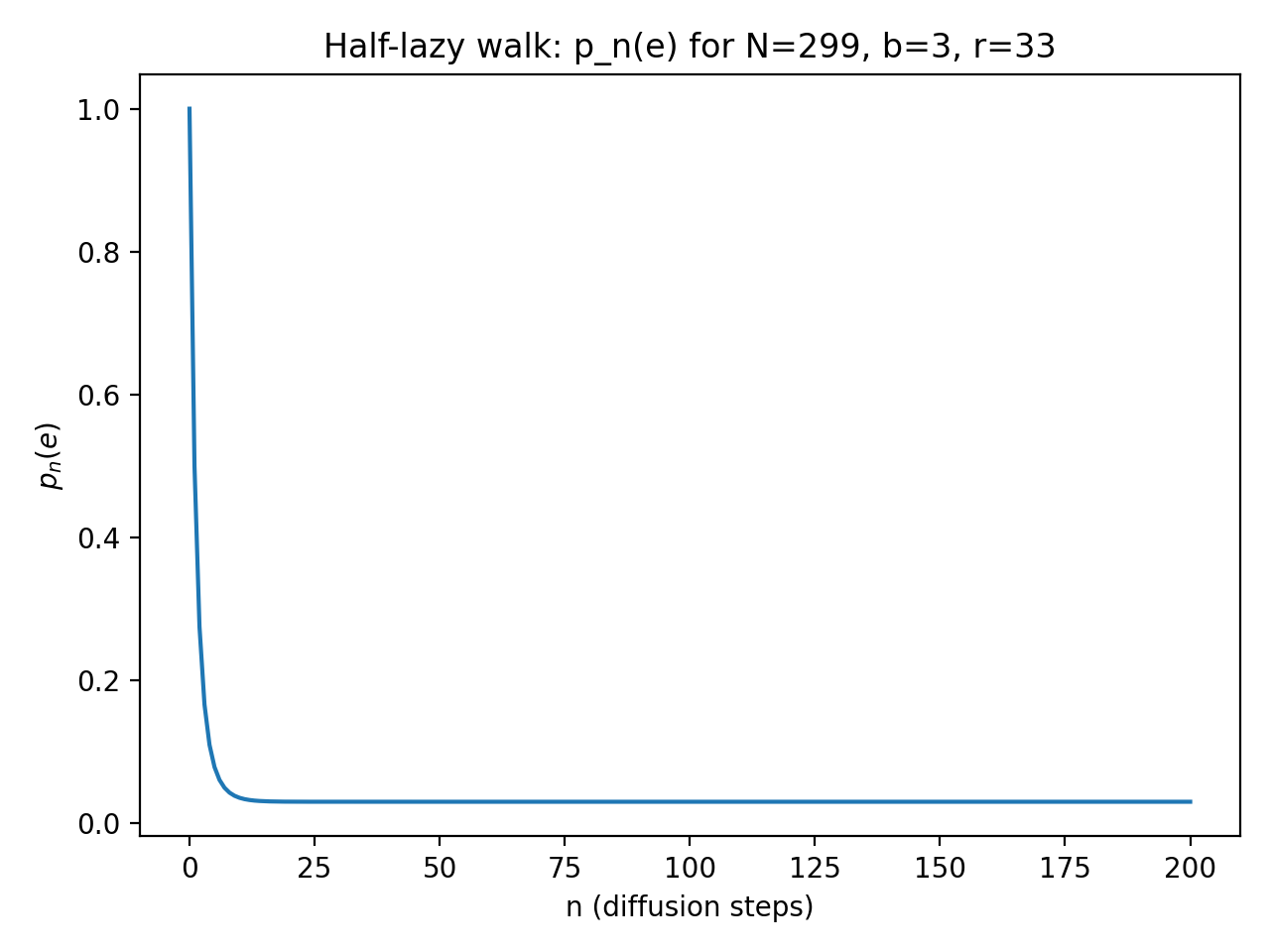}
  \caption{The identity value $p_n(e)$ for the half--lazy walk on $\langle 3\rangle$ (with $N=299$).}
  \label{fig:pid-299}
\end{figure}

\medskip
\noindent\textbf{Recovering the order by rounding.}
As shown in Figure~\ref{fig:pid-299}, $p_n(e)$ rapidly approaches $1/r$, so that $1/p_n(e)$ approaches $r=33$.
Figure~\ref{fig:invpid-299} plots $1/p_n(e)$ together with the target level $r$.
Numerically, rounding stabilizes quickly: from $n=17$ onward one has
\[
  \operatorname{round}\!\bigl(1/p_n(e)\bigr)=33.
\]
\begin{figure}[H]
  \centering
  \includegraphics[width=0.5\textwidth]{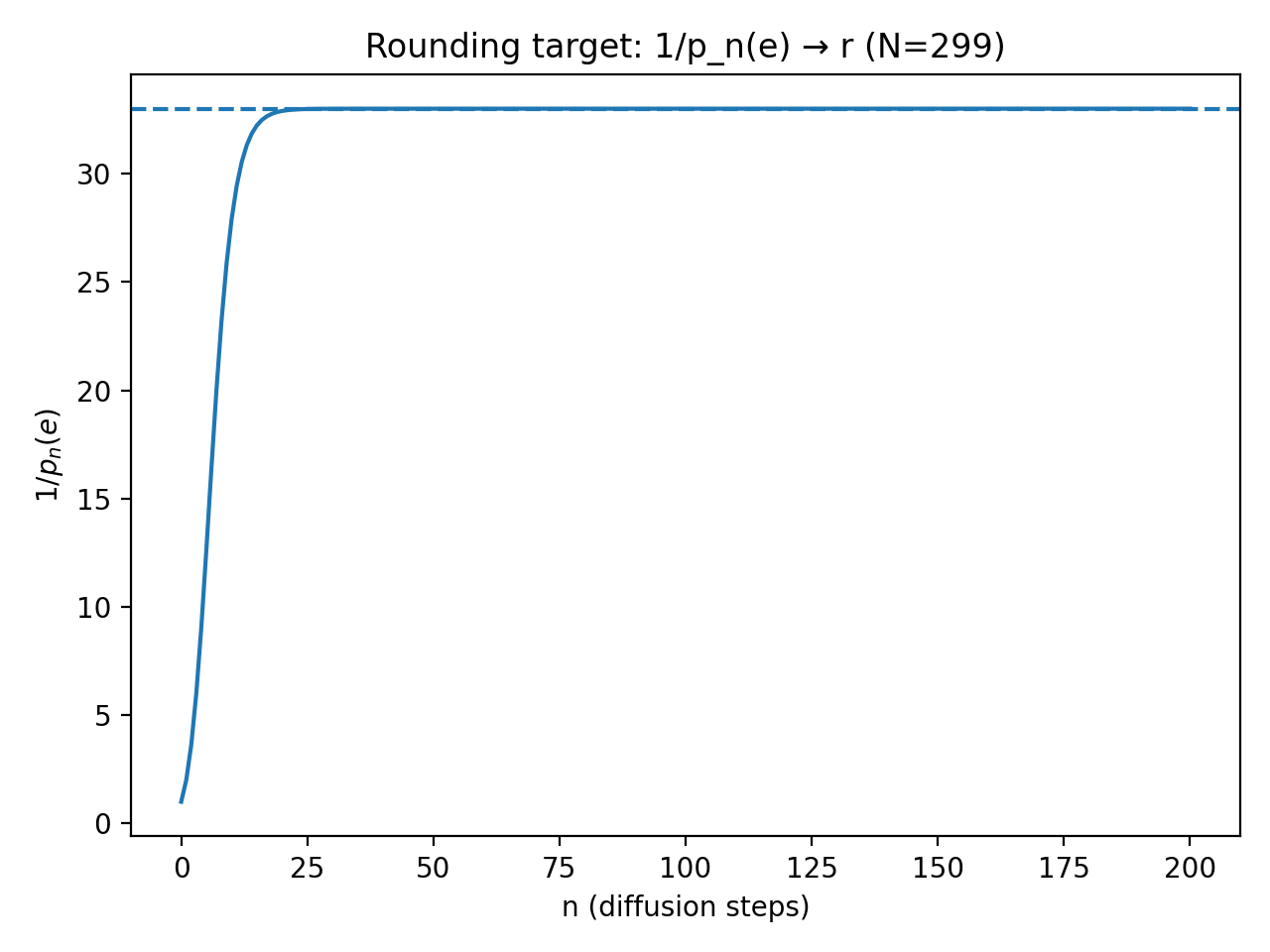}
  \caption{The sequence $1/p_n(e)$ converging to $r=\ord_{299}(3)=33$; rounding recovers $r$.
  (The dashed line indicates $r=33$.)}
  \label{fig:invpid-299}
\end{figure}

\noindent
This provides a concrete ``small-$N$'' visualization of the diffusion model:
the dynamics mix toward the uniform distribution on $\langle b\rangle$ (of size $r$),
and a single scalar readout $p_n(e)$ determines $r$ once it is close enough to $1/r$.

\subsection{Example 2: $N=1022117$}
\label{sec:example-1009-1013}

Now, we take  $N = 1009\cdot 1013 = 1022117$, and
$b = 576 \in (\mathbb{Z}/N\mathbb{Z})^\ast$.
One checks that $\gcd(b,N)=1$, and the order of $b$ modulo $N$ is odd:
\[
r=\ord_N(b)=5313.
\]
We run the half-lazy diffusion on the weighted Cayley graph of the cyclic subgroup
$\langle b\rangle$, using generators $b^{\pm 2^t}$ for $0\le t\le M$ with
$M=\lfloor \log_2 N\rfloor+1$.
The theory predicts that after $n=O((\log_2 N)^2)$ diffusion steps, the single heat-kernel
value at the identity determines $r$ by rounding $1/p_n(e)$. See Figures~\ref{fig:p-identity-1022117} and~\ref{fig:inv-p-identity-1022117}.

\medskip
\begin{remark}
The diffusion update is local, so an implementation does \emph{not} need to pre-construct the entire Cayley graph.  For visualization, however, we draw only a small neighborhood
of the identity (a local view), since the full graph has $r=5313$ vertices. See Figure~\ref{fig:cayley-1022117}.
\end{remark}
\begin{figure}[H]
  \centering
  \includegraphics[width=0.8\linewidth]{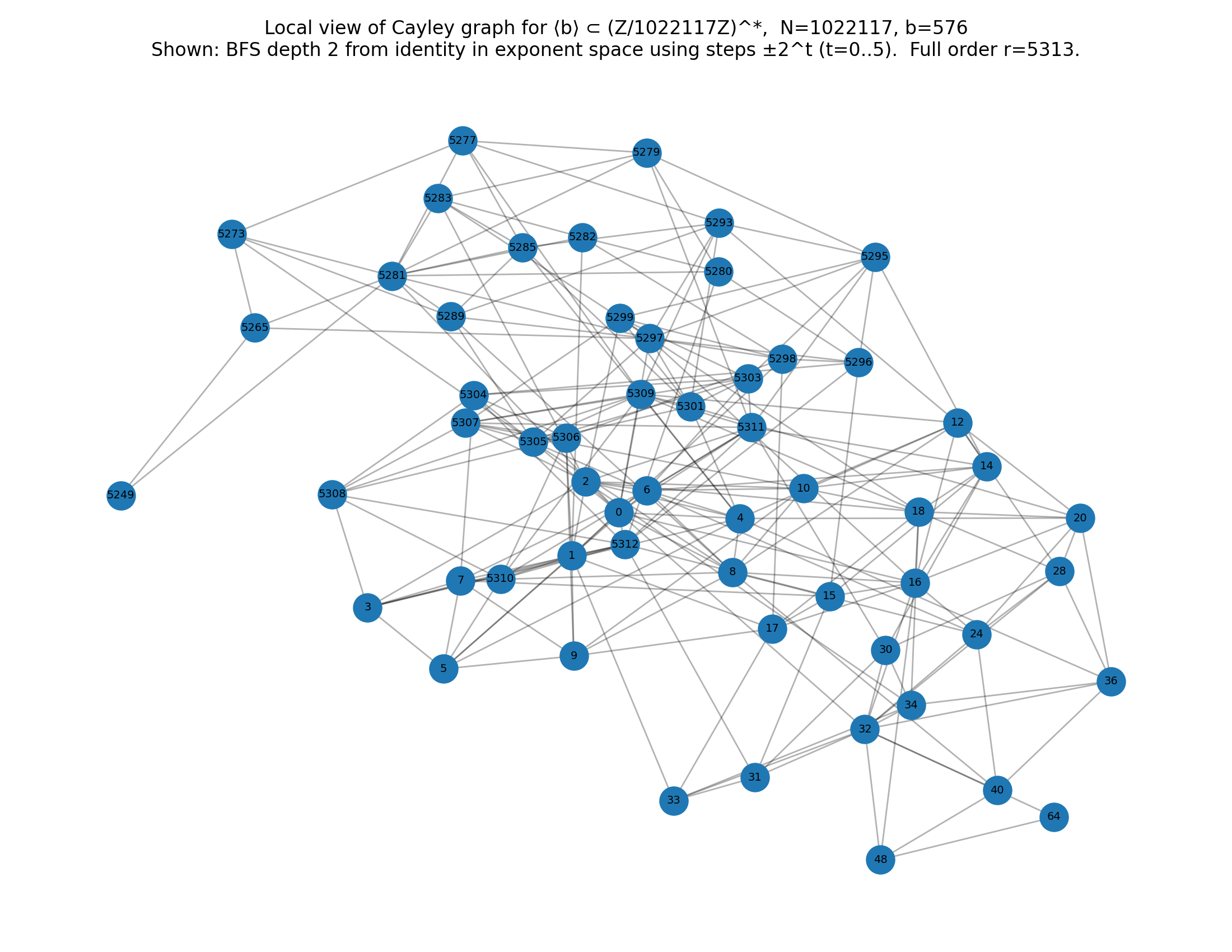}
  \caption{A \emph{local} view of the Cayley graph near the identity (shown in exponent coordinates).
  This is only for illustration; the full Cayley graph has $r=5313$ vertices.}
  \label{fig:cayley-1022117}
\end{figure}

\begin{figure}[H]
  \centering
  \includegraphics[width=0.7\linewidth]{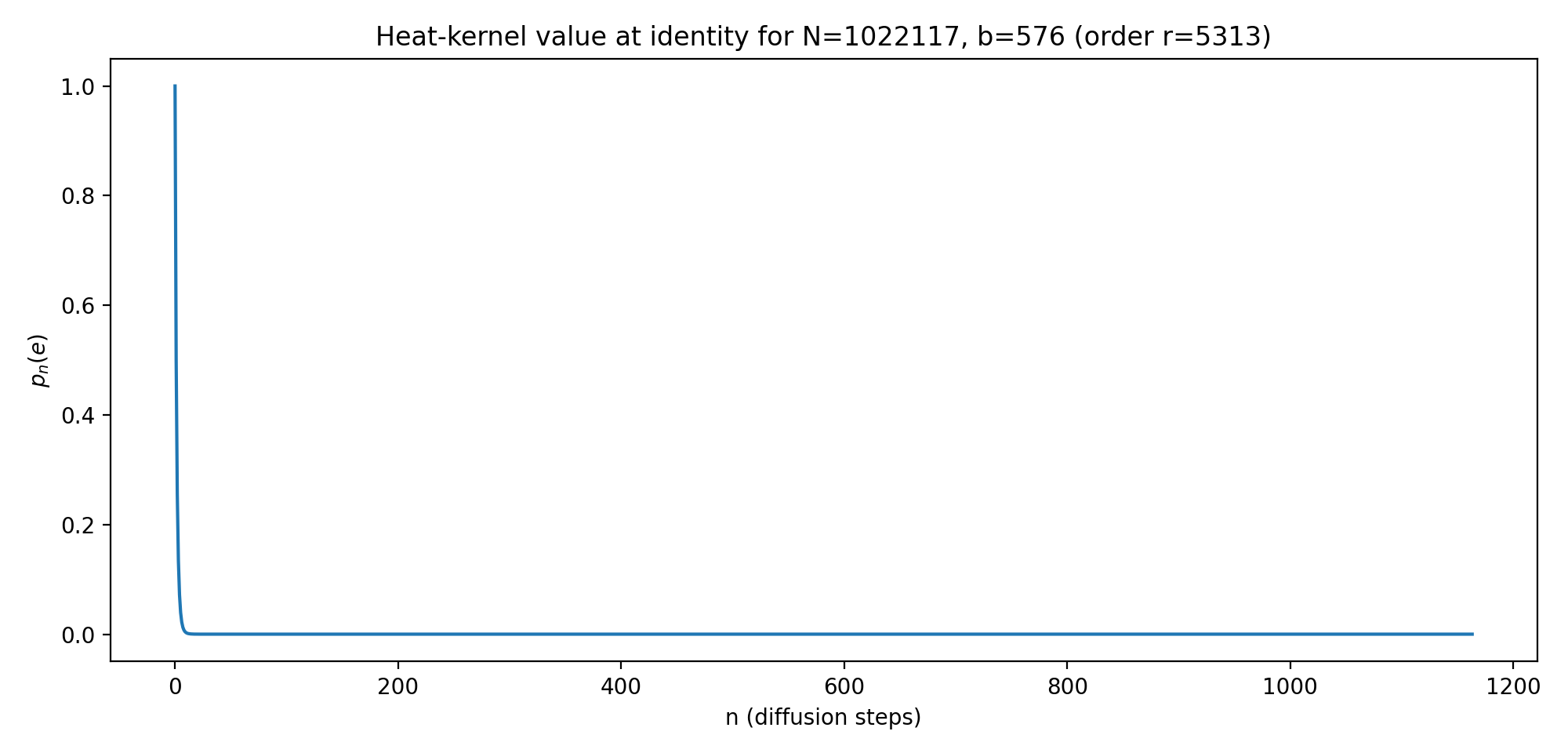}
  \caption{The identity heat-kernel value $p_n(e)$ as a function of the number $n$ of diffusion steps.
  As $n\to\infty$ one has $p_n(e)\to 1/r$.}
  \label{fig:p-identity-1022117}
\end{figure}

\begin{figure}[H]
  \centering
  \includegraphics[width=0.7\linewidth]{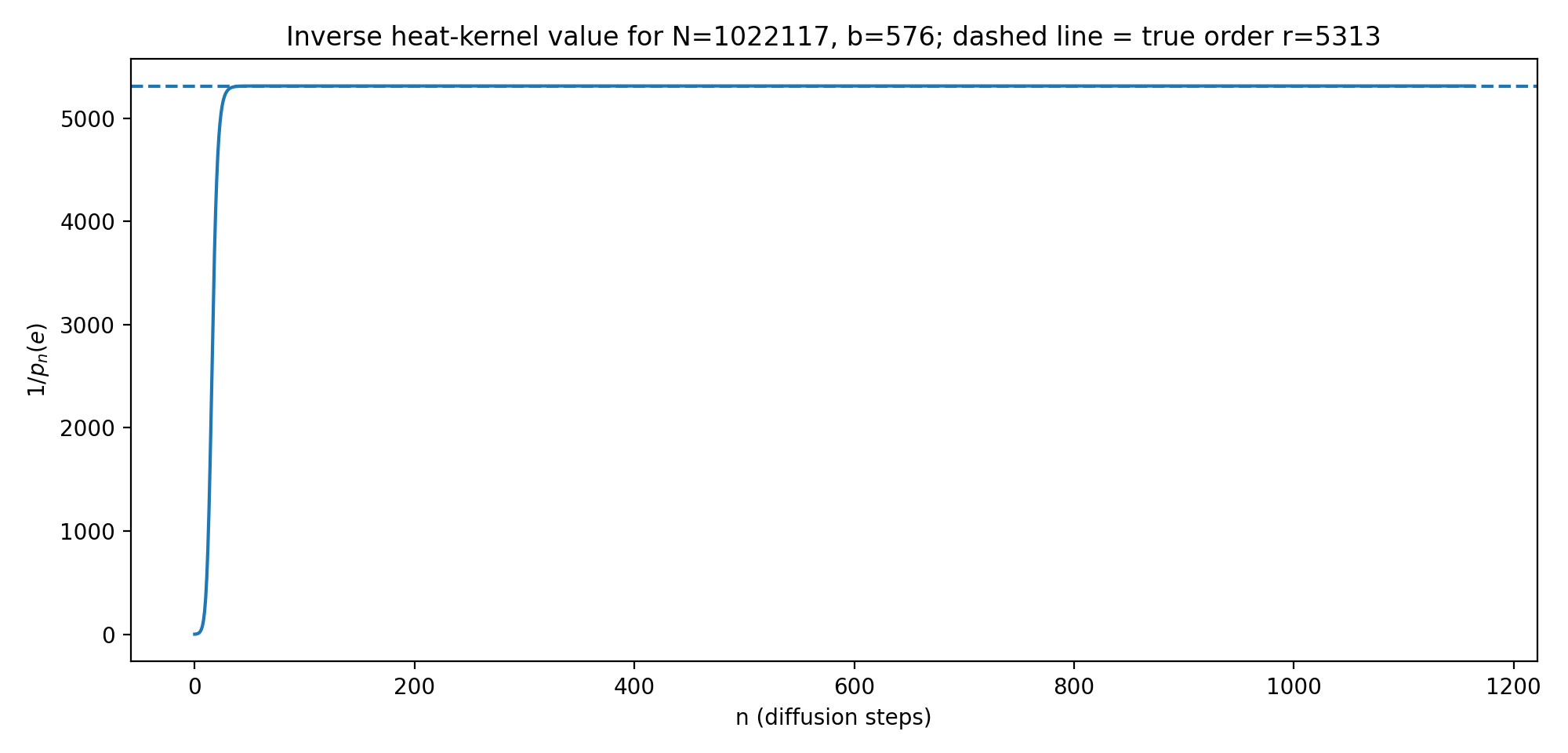}
  \caption{The inverse $1/p_n(e)$ converging to the true order $r=5313$ (dashed line).
  In this run, rounding $1/p_n(e)$ recovers $r$ after relatively few diffusion steps.}
  \label{fig:inv-p-identity-1022117}
\end{figure}

\subsection{Example 3: $
F_5 \;=\; 2^{2^5}+1$}
\label{E3}

We demonstrate the cycle-search implementation on the Fermat number
\[
F_5 \;=\; 2^{2^5}+1 \;=\; 2^{32}+1 \;=\; 4294967297.
\]
\[
\langle a\rangle \subset (\mathbb Z/N\mathbb Z)^{*}
\]
distinct words land at the same endpoint,
\[
a^{E_{\mathrm{new}}}\equiv a^{E_{\mathrm{prev}}}\pmod N,
\]
we record a \emph{cycle certificate}. This yields an exponent difference
\[
D \;=\; E_{\mathrm{new}}-E_{\mathrm{prev}}
\qquad\text{with}\qquad
a^{D}\equiv 1\pmod N,
\]
so $D$ is a multiple of the order $r=\ord_N(a)$. We then apply the standard ``halve when possible''
reduction to obtain a smaller multiple $D_{\min}$, and we maintain the running gcd
\[
g \;\leftarrow\; \gcd(g, D_{\min}).
\]
In practice, after a modest number of collisions this gcd often stabilizes to an order multiple,
which can then be reduced to the true order and fed into the usual order-to-factor step.

For $N=F_5$, using word length $L=2000$, at most $120000$ samples, and stability threshold $8$,
the algorithm succeeded on the first attempt, with random base
\[
a=3945765912.
\]
The successive cycle results and the evolution of the running gcd were:

\begin{verbatim}
[attempt 1] trying a = 3945765912
  sampling words of length L = 2000, max_samples = 120000, stable_hits = 8
[collision #  1]  D_min = 314919552   running_gcd = 314919552
[collision #  2]  D_min = 32543920512   running_gcd = 6700416
[collision #  3]  D_min = 52336949376   running_gcd = 6700416
[collision #  4]  D_min = 22975726464   running_gcd = 6700416
[collision #  5]  D_min = 40839035520   running_gcd = 6700416
[collision #  6]  D_min = 25012652928   running_gcd = 6700416
[collision #  7]  D_min = 4187760000   running_gcd = 6700416
[collision #  8]  D_min = 3986747520   running_gcd = 6700416
[collision #  9]  D_min = 18097823616   running_gcd = 6700416
[collision # 10]  D_min = 9749105280   running_gcd = 6700416
  stabilized gcd = 6700416
  reduced order r = 6700416
\end{verbatim}

After the second collision, the running gcd collapses to $g=6700416$ and remains unchanged; the
stabilization criterion is met after a few more collisions. The subsequent deterministic reduction
confirms that
\[
r=\ord_N(a)=6700416.
\]

Finally, since $r$ is even and $a^{r/2}\not\equiv \pm 1\pmod N$, the standard order-to-factor step
produces nontrivial divisors:
\[
\gcd\!\bigl(a^{r/2}-1,\,N\bigr)=6700417,
\qquad
\gcd\!\bigl(a^{r/2}+1,\,N\bigr)=641,
\]
and hence
\[
F_5 \;=\; 4294967297 \;=\; 641\cdot 6700417.
\]
In this run, the observed the clock time was $11.734$ seconds.

\subsection{Example 4: $N=8219999$}

A recently developed quantum-annealing approach embeds a compact binary
multiplier circuit into the Pegasus topology and reports factoring
\[
N=8219999 = 251\cdot 32749
\]
on a D-Wave Advantage 4.1 annealer; see~\cite{DSS24}.

We apply the diffusion-with-cycles implementation to the same integer $N$.
The algorithm explores the Cayley graph of the cyclic subgroup
\[
\langle a\rangle \subset (\mathbb Z/N\mathbb Z)^{\times}
\]
using dyadic generators $a^{\pm 2^t}$.
Whenever two distinct words land at the same endpoint,
\[
a^{E_{\mathrm{new}}}\equiv a^{E_{\mathrm{prev}}}\pmod N,
\]
we obtain a cycle certificate with exponent difference
$D=E_{\mathrm{new}}-E_{\mathrm{prev}}$ and hence $a^D\equiv 1\pmod N$.
After ``halving when possible'' to get $D_{\min}$, we update a running gcd
$g\leftarrow \gcd(g,D_{\min})$; once $g$ stabilizes, it is reduced to the true order
$r=\ord_N(a)$ and used in the standard order-to-factor step.

For this run we used word length $L=2000$, at most $120000$ samples, and stability threshold $8$.
The algorithm succeeded on the first attempt with base $a=7081686$:

\begin{verbatim}
[attempt 1] trying a = 7081686
  sampling words of length L = 2000, max_samples = 120000, stable_hits = 8
[collision #  1]  D_min = 12962750   running_gcd = 12962750
[collision #  2]  D_min = 111206750   running_gcd = 682250
[collision #  3]  D_min = 119393750   running_gcd = 682250
[collision #  4]  D_min = 42981750   running_gcd = 682250
[collision #  5]  D_min = 3411250   running_gcd = 682250
[collision #  6]  D_min = 3411250   running_gcd = 682250
[collision #  7]  D_min = 130309750   running_gcd = 682250
[collision #  8]  D_min = 68907250   running_gcd = 682250
[collision #  9]  D_min = 104384250   running_gcd = 682250
[collision # 10]  D_min = 55262250   running_gcd = 682250
  stabilized gcd = 682250
  reduced order r = 682250

SUCCESS: N = 32749 * 251

FINAL: 8219999 = 251 * 32749
TOTAL TIME: 3.120 s  (0:00:03)
\end{verbatim}

Thus, on this instance the cycle method recovers the factorization
\[
8219999 = 251\cdot 32749
\]
in about $3.12$ seconds of clock time in the reported run.

\subsection{Example 5: $N=1099551473989$}
In work involving Zapata Computing and collaborators, the integer \(N=1{,}099{,}551{,}473{,}989\) was factored using the variational
quantum factoring (VQF) workflow on a superconducting quantum processor; see \cite{Karamlou2021VQF}.  The following is the results
of the collision-strategy.

\begin{verbatim}
Enter N (odd, >3, not prime, not prime power), up to ~10 digits recommended: 1099551473989
Choose word length L (e.g. 800–2000): 2000
Choose max_samples per attempt (e.g. 30000–120000): 120000
Choose max_attempts (e.g. 20–80): 80

[attempt 1] trying a = 28213600916
  sampling words of length L = 2000, max_samples = 120000, stable_hits = 8
  no stabilized gcd from loops in this attempt (try another a).

[attempt 2] trying a = 1010844181454
  sampling words of length L = 2000, max_samples = 120000, stable_hits = 8
  no stabilized gcd from loops in this attempt (try another a).

[attempt 3] trying a = 45342608514
  sampling words of length L = 2000, max_samples = 120000, stable_hits = 8
[collision #  1]  D_min = 33192646812150   running_gcd = 33192646812150
  no stabilized gcd from loops in this attempt (try another a).

[attempt 4] trying a = 528571176770
  sampling words of length L = 2000, max_samples = 120000, stable_hits = 8
  no stabilized gcd from loops in this attempt (try another a).

[attempt 5] trying a = 591840418317
  sampling words of length L = 2000, max_samples = 120000, stable_hits = 8
  no stabilized gcd from loops in this attempt (try another a).

[attempt 6] trying a = 327685619692
  sampling words of length L = 2000, max_samples = 120000, stable_hits = 8
  no stabilized gcd from loops in this attempt (try another a).

[attempt 7] trying a = 1046655708948
  sampling words of length L = 2000, max_samples = 120000, stable_hits = 8
  no stabilized gcd from loops in this attempt (try another a).

[attempt 8] trying a = 477422510912
  sampling words of length L = 2000, max_samples = 120000, stable_hits = 8
  no stabilized gcd from loops in this attempt (try another a).

[attempt 9] trying a = 378381310618
  sampling words of length L = 2000, max_samples = 120000, stable_hits = 8
  no stabilized gcd from loops in this attempt (try another a).

[attempt 10] trying a = 747483761123
  sampling words of length L = 2000, max_samples = 120000, stable_hits = 8
  no stabilized gcd from loops in this attempt (try another a).

[attempt 11] trying a = 1001492753655
  sampling words of length L = 2000, max_samples = 120000, stable_hits = 8
  no stabilized gcd from loops in this attempt (try another a).

[attempt 12] trying a = 394688788349
  sampling words of length L = 2000, max_samples = 120000, stable_hits = 8
  no stabilized gcd from loops in this attempt (try another a).

[attempt 13] trying a = 478211699347
  sampling words of length L = 2000, max_samples = 120000, stable_hits = 8
  no stabilized gcd from loops in this attempt (try another a).

[attempt 14] trying a = 365250373964
  sampling words of length L = 2000, max_samples = 120000, stable_hits = 8
  no stabilized gcd from loops in this attempt (try another a).

[attempt 15] trying a = 893381186972
  sampling words of length L = 2000, max_samples = 120000, stable_hits = 8
  no stabilized gcd from loops in this attempt (try another a).

[attempt 16] trying a = 750796458253
  sampling words of length L = 2000, max_samples = 120000, stable_hits = 8
[collision #  1]  D_min = 3966231680600   AGGRESSIVE ONE-COLLISION FACTOR: 1048589 * 1048601

FINAL: 1099551473989 = 1048589 * 1048601
TOTAL TIME: 3132.641 s  (0:52:13)
\end{verbatim}

\section{A minimal RC-network implementation of the diffusion primitive}
\label{7}

To conclude, let us
briefly describe a possible physical realization of a diffusion primitive in \emph{continuous time}
and explain why sampling at a small time step $\Delta t$ approximates the discrete operator
used in our model.

Fix a finite undirected weighted graph with vertex set $V=\{1,\dots,r\}$ and symmetric
conductances $g_{ij}=g_{ji}\ge 0$ ($g_{ij}=0$ if there is no edge).
Attach to each vertex $i$ a capacitor $C>0$ to ground, and connect vertices $i$ and $j$
by a resistor whose conductance is $g_{ij}=1/R_{ij}$.
Let $V_i(t)$ be the voltage at node $i$ at time $t$, and set
\[
V(t)=(V_1(t),\dots,V_r(t))^{\mathsf T}.
\]

By Kirchhoff's current law, the capacitor current equals minus the net resistive outflow:
\[
C\,\frac{d}{dt}V_i(t) = -\sum_{j=1}^r g_{ij}\bigl(V_i(t)-V_j(t)\bigr).
\]
Let $L$ be the associated weighted graph Laplacian, meaning that
\[
L_{ii}=\sum_{j\ne i} g_{ij},
\qquad
L_{ij}=-g_{ij}\ \ (i\ne j).
\]
Then the network dynamics are the linear ODE
\begin{equation}\label{eq:cont-diffusion}
C\,\frac{d}{dt}V(t) = -L\,V(t),
\qquad\text{hence}\qquad
V(t)=\exp\!\left(-\frac{t}{C}L\right)V(0).
\end{equation}
Thus the $RC$ network implements the continuous-time heat flow on the graph.

In the abstract model of this paper, one diffusion step is a fixed linear operator
\[
W:\mathbb{R}^r\to\mathbb{R}^r
\]
applied repeatedly.  A continuous-time implementation produces instead the semigroup
$t\mapsto \exp(-(t/C)L)$, and a natural discrete-time primitive is obtained by sampling at
a fixed time increment $\Delta t>0$:
\begin{equation}\label{eq:W-sampling}
V^{(n+1)} = V((n+1)\Delta t)
= \exp\!\left(-\frac{\Delta t}{C}L\right) V(n\Delta t)
= W_{\Delta t}\,V^{(n)},
\end{equation}
where
\[
W_{\Delta t}=\exp\!\left(-\frac{\Delta t}{C}L\right).
\]
This is an \emph{exact} discrete-time evolution obtained by observing the continuous circuit
only at times $0,\Delta t,2\Delta t,\dots$.

For sufficiently small $\Delta t$, the matrix exponential admits the first-order expansion
\begin{equation}\label{eq:expansion}
\exp\!\left(-\frac{\Delta t}{C}L\right)
= I-\frac{\Delta t}{C}L + O(\Delta t^2),
\qquad \Delta t\to 0,
\end{equation}
with the error understood entrywise or in operator norm.
Consequently, one sampled step satisfies
\[
V^{(n+1)}-V^{(n)}
= -\frac{\Delta t}{C}L\,V^{(n)} + O(\Delta t^2).
\]
In words: over a very short time interval, each node voltage changes by a small amount
proportional to the Laplacian (a weighted difference between the node and its neighbors),
which is the defining local averaging mechanism of diffusion.

Equivalently, if one specifies a target discrete-time averaging operator of the form
\[
P = I-\gamma L
\]
for some scale $\gamma>0$, then choosing
\[
\gamma=\frac{\Delta t}{C}
\]
makes the circuit's sampled step satisfy
\[
W_{\Delta t}=\exp\!\left(-\frac{\Delta t}{C}L\right)=P+O(\Delta t^2).
\]
Therefore, by taking $\Delta t$ small (and correspondingly increasing the number of sampled
steps to reach a fixed physical mixing time), the continuous-time $RC$ diffusion approximates
the discrete diffusion process used in our analysis.

\subsubsection*{Implementation}
In a physical realization, the circuit runs continuously, and one ``diffusion step'' in the
computational accounting corresponds to either:
\begin{itemize}
\item sampling the node voltages every $\Delta t$ seconds (giving the exact map $W_{\Delta t}$), or
\item operating in a regime where $\Delta t$ is small so that each sample is a small local
averaging update, approximating a prescribed discrete operator.
\end{itemize}
This provides a concrete interpretation of the diffusion primitive as a hardware-supported
local relaxation step.

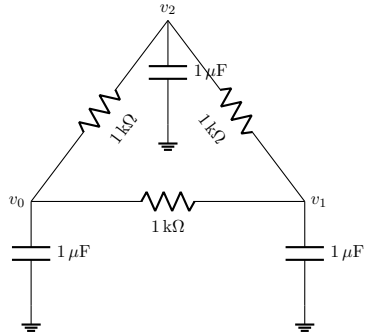
\begin{figure}[H]
\centering
\begin{circuitikz}[american, scale=0.6, transform shape]
  \coordinate (v0) at (0,0);
  \coordinate (v1) at (6,0);
  \coordinate (v2) at (3,4);

  \draw (v0) to[R, l_=$1\,\mathrm{k}\Omega$] (v1);
  \draw (v1) to[R, l^=$1\,\mathrm{k}\Omega$] (v2);
  \draw (v2) to[R, l^=$1\,\mathrm{k}\Omega$] (v0);

  \draw (v0) to[C, l=$1\,\mu\mathrm{F}$] ++(0,-2.3) node[ground]{};
  \draw (v1) to[C, l=$1\,\mu\mathrm{F}$] ++(0,-2.3) node[ground]{};
  \draw (v2) to[C, l=$1\,\mu\mathrm{F}$] ++(0,-2.3) node[ground]{};

  \node[left]  at (v0) {$v_0$};
  \node[right] at (v1) {$v_1$};
  \node[above] at (v2) {$v_2$};

\end{circuitikz}
\caption{RC network implementing diffusion on the Cayley graph for $N=21$.
Take $N=21$ and choose $a=2$. With $M=\lfloor \log_2 21\rfloor+1=5$, set
$b\equiv a^{2^{M}}\equiv 2^{32}\pmod{21}$.
Since $\ord_{21}(2)=6$, we have $2^{32}\equiv 2^{2}\equiv 4\pmod{21}$, hence
$b=4$ and $\ord_{21}(4)=3$.
Therefore $G=\langle 4\rangle=\{1,4,16\}$ has three vertices, and the generators
$b^{\pm 2^t}$ reduce to steps $\pm 1$ on $\Z/3\Z$, so the Cayley graph is a $3$-cycle.}

\label{fig:rc-21}
\end{figure}

\section*{Acknowledgments}
Carlos A. Cadavid gratefully acknowledges the financial support of Universidad EAFIT (Colombia) for the project Study and
Applications of Diffusion Processes of Importance in Health and Computation (project code 11740052022).
Jay Jorgenson acknowledges grant support from PSC-CUNY Award
68462-00 56, which is jointly funded
by the Professional Staff Congress and The City University of New York.
Juan D. Vélez gratefully acknowledges the Universidad Nacional de Colombia for its support during this research.

\section*{Author information}

Carlos A.\ Cadavid\\
Department of Mathematics, Universidad Eafit, Medell\'in, Colombia\\
\texttt{ccadavid@eafit.edu.co}

\medskip
Paulina Hoyos\\
Department of Mathematics, The University of Texas at Austin, USA\\
\texttt{paulinah@utexas.edu}

\medskip
Jay Jorgenson\\
Department of Mathematics, The City College of New York, USA\\
\texttt{jjorgenson@mindspring.com}

\medskip
Lejla Smajlovi\'c\\
Department of Mathematics, University of Sarajevo, Bosnia and Herzegovina\\
\texttt{lejlas@pmf.unsa.ba}

\medskip
Juan D.\ V\'elez\\
Department of Mathematics, Universidad Nacional de Colombia, Medell\'in, Colombia\\
\texttt{jdvelez@unal.edu.co}


\begin{thebibliography}{99}



\bibitem[AKS04]{AKS04}
M.~Agrawal, N.~Kayal, and N.~Saxena,
\emph{PRIMES is in P},
Ann.\ of Math.\ (2) \textbf{160} (2004), 781--793.

\bibitem[Ba79]{Ba79}
L.~Babai,
\emph{Spectra of Cayley graphs},
J.\ Combin.\ Theory Ser.\ B \textbf{27} (1979), no.~2, 180--189.

\bibitem[BS96]{BachShallit}
E.~Bach and J.~Shallit,
\emph{Algorithmic Number Theory, Volume 1: Efficient Algorithms},
MIT Press, 1996.

\bibitem[Be07]{Be07}
D.~J.~Bernstein, H.~W.~Lenstra, and J.~Pila,
\emph{Detecting perfect powers by factoring into coprimes},
Math.\ Comp.\ \textbf{76} (2007), 385--388.

\bibitem[CBK22]{CBK22}
D.~Chicayban Bastos and L.~A.~Kowada,
\emph{A quantum version of Pollard's Rho of which Shor's algorithm is a particular case},
in: \emph{Computing and Combinatorics},
Lecture Notes in Comput.\ Sci.\ \textbf{13595}, Springer, Cham, 2022, 212--219.

\bibitem[CHJSV23]{CHJSV23}
C.~A.~Cadavid, P.~Hoyos, J.~Jorgenson, L.~Smajlovi\'c, and J.~D.~V\'elez,
\emph{Discrete diffusion-type equation on regular graphs and its applications},
J.\ Difference Equ.\ Appl.\ \textbf{29} (2023), no.~4, 455--488.

\bibitem[Ch97]{Chung97}
F.~R.~K.~Chung,
\emph{Spectral Graph Theory},
American Mathematical Society, 1997.

\bibitem[CP05]{CrandallPomerance}
R.~Crandall and C.~Pomerance,
\emph{Prime Numbers: A Computational Perspective},
2nd ed., Springer, 2005.

\bibitem[CR62]{CR62}
C.~W.~Curtis and I.~Reiner,
\emph{Representation Theory of Finite Groups and Associative Algebras},
Wiley, 1962; reprint: Wiley Classics Library, 1988.

\bibitem[DSS24]{DSS24}
J.~Ding, G.~Spallitta, and R.~Sebastiani,
\emph{Effective prime factorization via quantum annealing by modular locally-structured embedding},
Sci.\ Rep.\ \textbf{14} (2024), Article~3518.
doi:\,10.1038/s41598-024-53708-7.

\bibitem[EJ96]{EkertJozsa96}
A.~K.~Ekert and R.~Jozsa,
\emph{Quantum computation and Shor's factoring algorithm},
Rev.\ Mod.\ Phys.\ \textbf{68} (1996), no.~3, 733--753.

\bibitem[HW08]{HardyWright}
G.~H.~Hardy and E.~M.~Wright,
\emph{An Introduction to the Theory of Numbers},
6th ed., Oxford University Press, 2008.

\bibitem[Hh25]{Hh25}
M.~Hhan,
\emph{A new approach to generic lower bounds: classical/quantum MDL, quantum factoring, and more},
in: \emph{Advances in Cryptology---EUROCRYPT 2025. Part VII},
Lecture Notes in Comput.\ Sci.\ \textbf{15607}, Springer, Cham, 2025, 345--374.

\bibitem[KM12]{KM12}
J.~Kaczorowski and G.~Molteni,
\emph{Extremal values for the sum $\sum_{r=1}^{\tau} e(a2^r/q)$},
J.\ Number Theory \textbf{132} (2012), 2595--2603.

\bibitem[Kn97]{Knuth2}
D.~E.~Knuth,
\emph{The Art of Computer Programming, Volume 2: Seminumerical Algorithms},
3rd ed., Addison--Wesley, 1997.

\bibitem[KSK21]{Karamlou2021VQF}
A.~H.~Karamlou, W.~A.~Simon, A.~Katabarwa, T.~L.~Scholten, B.~Peropadre, Y.~Cao, et~al.,
\emph{Analyzing the performance of variational quantum factoring on a superconducting quantum processor},
npj Quantum Information \textbf{7} (2021), Article 156.
DOI: 10.1038/s41534-021-00478-z.

\bibitem[KSV99]{KSV99}
A.~Yu.~Kitaev, A.~H.~Shen, and M.~N.~Vyalyi,
\emph{Classical and Quantum Computation},
Graduate Studies in Mathematics \textbf{47}, American Mathematical Society, Providence, RI, 2002
(translated from the 1999 Russian original by L.~J.~Senechal).

\bibitem[Leh00]{Leh00}
D.~N.~Lehmer,
\emph{Asymptotic evaluation of certain totient sums},
Amer.\ J.\ Math.\ \textbf{22} (1900), 293--335.

\bibitem[LPW09]{LPW09}
D.~A.~Levin, Y.~Peres, and E.~L.~Wilmer,
\emph{Markov Chains and Mixing Times},
American Mathematical Society, 2009.

\bibitem[Me96]{MenezesBook}
A.~J.~Menezes, P.~C.~van Oorschot, and S.~A.~Vanstone,
\emph{Handbook of Applied Cryptography},
CRC Press, 1996.

\bibitem[NC10]{NC10}
M.~A.~Nielsen and I.~L.~Chuang,
\emph{Quantum Computation and Quantum Information},
10th Anniversary Edition, Cambridge University Press, 2010.

\bibitem[Ni18]{Ni18}
B.~Nica,
\emph{A Brief Introduction to Spectral Graph Theory},
EMS Textbooks in Mathematics, European Mathematical Society, 2018.

\bibitem[Nym72]{Nym72}
J.~E.~Nymann,
\emph{On the probability that $k$ positive integers are relatively prime},
J.\ Number Theory \textbf{4} (1972), 469--473.

\bibitem[Ra80]{Ra80}
M.~O.~Rabin,
\emph{Probabilistic algorithm for testing primality},
J.\ Number Theory \textbf{12} (1980), no.~1, 128--138.

\bibitem[Ra24]{Ra24}
S.~Ragavan and V.~Vaikuntanathan,
\emph{Space-efficient and noise-robust quantum factoring},
in: \emph{Advances in Cryptology---CRYPTO 2024. Part VI},
Lecture Notes in Comput.\ Sci.\ \textbf{14925}, Springer, Cham, 2024, 107--140.

\bibitem[Re25]{Re25}
O.~Regev,
\emph{An efficient quantum factoring algorithm},
J.\ ACM \textbf{72} (2025), no.~1, Art.~10, 13 pp.
doi:\,10.1145/3708471.

\bibitem[Sh94]{Shor94}
P.~W.~Shor,
\emph{Algorithms for quantum computation: Discrete logarithms and factoring},
in: \emph{Proceedings of the 35th Annual Symposium on Foundations of Computer Science (FOCS)},
IEEE, 1994, 124--134.

\bibitem[Sh97]{Sh97}
P.~W.~Shor,
\emph{Polynomial-time algorithms for prime factorization and discrete logarithms on a quantum computer},
SIAM J.\ Comput.\ \textbf{26} (1997), no.~5, 1484--1509.

\bibitem[Va19]{Va19}
J.~Vandehey,
\emph{Differencing methods for Korobov-type exponential sums},
J.\ Anal.\ Math.\ \textbf{138} (2019), 405--439.

\bibitem[XQLM23]{XQLM23}
L.~Xiao, D.~Qiu, L.~Luo, and P.~Mateus,
\emph{Distributed Shor's algorithm},
Quantum Inf.\ Comput.\ \textbf{23} (2023), no.~1--2, 27--44.

\bibitem[Za13]{Za13}
P.~Zawadzki,
\emph{Closed-form formula on quantum factorization effectiveness},
Quantum Inf.\ Process.\ \textbf{12} (2013), 97--108.

\end{thebibliography}
\end{document}